\documentclass[12pt]{article}
\pdfoutput=1
\usepackage{Zoe_C_Style}
\usepackage{dsfont}
\usepackage{accents}

\newlength{\dtildeheight}
\newcommand{\doubletilde}[1]{%
	\settoheight{\dtildeheight}{\ensuremath{\tilde{#1}}}%
	\addtolength{\dtildeheight}{-0.30ex}%
	\tilde{\vphantom{\rule{1pt}{\dtildeheight}}%
		\smash{\tilde{#1}}}}
\newtheorem*{theorem*}{Theorem}

\title{A Cosheaf Theory of Reciprocal Figures: Planar and Higher Genus Graphic Statics}
\author{Zoe Cooperband, Robert Ghrist, Jakob Hansen}

\begin{document}
	\pagenumbering{arabic}
	\maketitle
	
	\begin{abstract}
		This paper introduces cellular sheaf theory to graphical methods and reciprocal 
		constructions in structural engineering. The elementary mechanics and statics of 
		trusses are derived from the linear algebra of sheaves and cosheaves. Further, the 
		homological algebra of these mathematical constructions cleanly and concisely describes the 
		formation of 2D reciprocal diagrams and 3D polyhedral lifts. Additional 
		relationships between geometric quantities of these dual diagrams are developed, 
		including systems of impossible edge rotations. These constructions generalize to 
		non-planar graphs. When a truss embedded in a torus or higher genus surface has a 
		sufficient degree of axial self stress, we show non-trivial reciprocal figures and 
		non-simply connected polyhedral lifts are guaranteed to exist.
	\end{abstract}
	
	%
	\section{Introduction}
	\label{sec:intro}
	%
	Graphic statics consists of a family of techniques for investigating the 
	static equilibrium solutions of structures geometrically. As early as the 18th century, 
	Varignon \cite{Varignon1725} introduced the concept of the {\it funicular polygon} and 
	a {\it polygon of forces} to maintain structural equilibrium within a loaded structure 
	\cite{ReciprocalAkbarzadeh2016}. Over a century later the foundational properties of
	{\it reciprocal figures} were further developed by Maxwell 
	\cite{ReciprocalMaxwell1864, maxwell_1870}, Cremona \cite{cremona1872figure, cremona1890graphical}, 
 	and Rankine \cite{rankine1864}, among others. These graphical 
	methods were widely used by designers such as Antoni Gaud\'i, Heinx Isler and Frei 
	Otto among others to find structurally efficient and beautiful structural forms \cite{Rasch1996}.
	
	In recent times graphical methods have experienced a
	renaissance of interest and attention, spurred on by, e.g., structural design
	\cite{fuhrimann2018data, akbarzadeh2016three},
	optimization \cite{beghini2014structural, Konstantatou2018}, and
	education \cite{Saliklis2019}. These methods have been utilized in Crapo and Whiteley's 
	study of rigidity theory \cite{Crapo1979, Crapo1993}, Calladine and Pellegrino's work on 
	tensegrity structures \cite{Calladine1978, Pellegrino1986}, and Tachi's work on flexible 
	origami structures \cite{Tachi2012}. The algebraic foundations of the theory have been 
	further solidified by Micheletti \cite{GeneralizedMicheletti2008}, Van Mele and Block 
	\cite{AlgebraicVanMele2014}, and many others. McRobie, Konstantatou, and
	collaborators have further developed the understanding of kinematics and virtual work 
	of reciprocal structures \cite{MechanismsMcRobie2016, McRobie2017a}. 
	
	Great strides have been made recently in the understanding and adoption of {\it 3D 
	graphic statics} \cite{ReciprocalAkbarzadeh2016}. This theory has been proven valuable 
	in the design and analysis of compressed concrete \cite{akbarzadeh2017prefab, 
	bolhassani2018structural} and glass structures \cite{akbarzadeh2019design}. With 
	the additional challenges brought by the extra dimension, many researchers have 
	advanced the mathematical understanding the construction of 
	\cite{AlgebraicHablicsek2019, GeometryMcRobie2017, PolaritiesKonstantatou2018} 
	and the mechanical properties of this spatial duality \cite{mcrobieMRReciprocal2016, 
	mcrobie2021stability}. 
	
	\subsection{Background}
	\label{sec:bg}
	Although the most interesting current work is happening in 3D graphic statics, it is 
	nevertheless true
	that any advances here must rest on a firm platform of 2D graphic statics. We therefore 
	focus
	on a novel approach to 2D graphic statics as a stepping stone towards future 
	foundational development in 3D.
	
	In 2D graphic statics, a finite abstract graph $G$ is embedded in the plane with edges 
	sent to straight lines, forming a mechanical truss with free-rotating joints known 
	collectively 
	as a {\it form diagram}. This truss may admit non-trivial axial self stresses---an
	assignment of an internal tension or compression force to each edge such that
	these forces mutually cancel over vertices. The combined free body 
	diagram of axial forces is an embedding of the dual graph $\tilde{G}$
	in the plane called a {\it reciprocal figure} or {\it force diagram}. The internal tension or 
	compression force vectors on the primal graph form the edges of the dual, pictured in 
	Figure~\ref{fig:planar_node}. Through this process, the internal stresses can be 
	visualized 
	and understood graphically, revealing the rich underlying linear algebraic relationships.
	
	In structural analysis, {\it local} quantities considered in static analysis, such as stresses 
	over edges, coalesce into global properties of the truss such as axial self stress. This 
	conglomeration of local data constrained by linear relationships is naturally described 
	by 
	the data structures known as {\it cellular sheaves and cosheaves} \cite{GhristEAT}. 
	These 
	are the discrete analogue of (continuous) sheaves, widely used in topology, algebraic 
	geometry, logic, and other mathematical fields \cite{bredon1997sheaf, MacLane1994}. 
	Initially developed in the theses of Shepard \cite{shepard1986cellular} and Curry
	\cite{Curry2013}, cellular sheaves (and their cosheaf duals) have found 
	applications in network coding \cite{ghrist2011network}, Reeb graphs 
	\cite{de2016categorified}, logic circuits \cite{robinson2012asynchronous}, and pursuit 
	and evasion games \cite{ghrist2017positive}. Exciting recent work has applied sheaves 
	towards persistent homology \cite{yoon2018cellular, yoon2020persistence, 
		russold2022persistent}, 
	distributed optimization \cite{hansen2019distributed}, graph neural networks 
	\cite{hansen2020sheaf, bodnar2022neural, barbero2022sheaf}, distributed consensus 
	and flocking \cite{Hansen2019}, opinion dynamics \cite{Hansen2019, 
		hansen2021opinion, ghrist2022network}, and lattice theory \cite{RiessLatticeTarski, 
		riess2022diffusion}. In this paper, we continue this program and extend previous 
		work \cite{TowardsCooperband2023} applying cellular sheaf theory to the geometric 
	structures encountered in graphic statics.
	
	Many of these recent applications of cellular sheaves to applied problems were 
	presaged by work decades earlier which did not have the simpler cellular theory 
	available. Excellent examples include Schapira's and Viro's independent works on Euler 
	calculus for sheaves of
	constructible functions with applications to image analysis \cite{viro1988some, 
		schapira1995tomography}. In similar fashion, Billera and Yuzvinsky independently 
	developed cosheaves of piecewise polynomial splines, commonly used in architecture, 
	design, and font systems \cite{billera1988homology, yuzvinsky1992modules}. In graphic 
	statics, sheaves were also explored as an approach decades ago.
	Beginning in the 1980s, Crapo \cite{Crapo1988, Crapo1995}
	used sheaves and homological methods to describe structural relations. Alongside 
	Crapo, Whiteley developed a cosheaf theory not unlike that described here, then 
	coined 
	{\it geometric homology} \cite{Whiteley1998}. Other matrix methods have been 
	developed to algebraically formulate reciprocity without the full groundwork of sheaves 
	and their cohomology \cite{AlgebraicVanMele2014, MechanismsMitchell2016, 
		MechanismsMcRobie2016}. Projective duality has been utilized in parts to develop 
	reciprocal relations \cite{Williams2016, Crapo1982, Whiteley1979}. Our contributions 
	via 
	cellular sheaf theory unify and extend these threads with a theory that is both simpler 
	and 
	more powerful.
	
	\begin{figure}[ht]\centering
		\begin{subfigure}[t]{0.45\textwidth}\centering
			\includegraphics{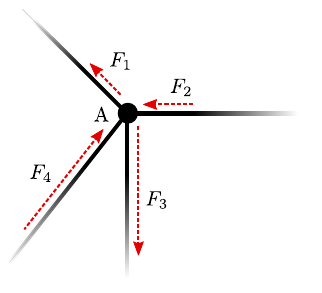}
			\caption{}
		\end{subfigure}
		\begin{subfigure}[t]{0.45\textwidth}\centering
			\includegraphics{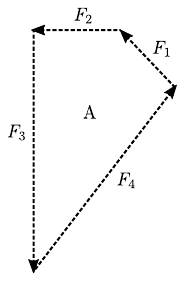}
			\caption{}
		\end{subfigure}
		\caption{The crux of reciprocity. Tension or compression forces along bars
			are transferred to node A where they sum to zero (a). When placed tip to
			tail using a clockwise orientation, these force vectors form the boundary of
			a dual polygon A, pictured in (b).}
		\label{fig:planar_node}
	\end{figure}
	
	
	\subsection{Outline and Contributions}
	\label{sec:contributions}
	This paper is a self-contained introduction to a cosheaf theory of 2D graphic statics. 
	This includes the statics of planar graphs, extended to a cellular structure on 
	the 2-sphere $S^2$ by filling in faces within graph cycles. Adapting to more complex 
	cellular topology beyond $S^2$ allows for the graphical analysis of non-planar graphs. 
	There are several nuances to this extension which are discussed after the grounding 
	planar theory is developed. These fundamental findings are expressed using the 
	homological algebra of cellular sheaves and cosheaves developed in 
	Section~\ref{sec:allsheaves}.
	
	In Section~\ref{sec:constructions}, we introduce the {\it constant cosheaf}, {\it force 
	cosheaf}, {\it linkage sheaf}, and {\it position sheaf}, which serve as the essential 
	foundation for the foundation of sheaf-theoretic graphic statics. 
	These constructions are purely in terms of vector spaces and linear maps, structured in 
	a manner that respects the cellular topology. Additionally, we demonstrate that {\it 
	Maxwell's counting rule} for frames can be interpreted as an instance of the cellular 
	cosheaf {\it Euler characteristic}, and the {\it stiffness matrix} is an instance of the  
	{\it sheaf (Hodge) Laplacian}.
	
	In Section~\ref{sec:2D1} two dimensional planar graphic statics is developed, where 
	Theorem~\ref{thm:plane2D} characterizes the homological relationship between a 
	planar graph and its dual graph embedding. We describe {\it reciprocity} between the 
	force cosheaf and position sheaf, revealing relations between a force diagram and its 
	reciprocal form diagrams. These connections extend the work of previous authors on 
	mechanical duality in planar graphic statics \cite{MechanismsMitchell2016, 
	MechanismsMcRobie2016}.
	
	\begin{theorem*}[\ref{thm:plane reciprocity},\ref{thm:maxwell_lift}]
		Suppose $(X, p)$ and $(\tilde{X}, q)$ are planar reciprocal diagrams over 
		the plane $\R^2$. The following vector spaces are isomorphic:
		\begin{itemize}
			\item[(i)] The space of mechanisms and global rotations of $(X,p)$.
			\item[(ii)] The space of parallel deformations of $(X,p)$ up to global translation.
			\item[(iii)] The space of impossible edge rotations over $(\tilde{X}, q)$.
			\item[(iv)] The space of axial self stresses over $(\tilde{X}, q)$.
			\item[(v)] The space of pure self shear stresses over $(\tilde{X}, q)$.
			\item[(vi)] The space of vertical polyhedral lifts of $(\tilde{X},q)$ up to shifts of a 
			global affine function.
		\end{itemize}
	\end{theorem*}
	
	We note that here $X$ and $\tilde{X}$ must consist of nonsingular polygons, where 
	each 
	polygonal cell is a topological disk with boundary consisting of lower dimensional 
	polygonal cells. This is a general condition on cells to disallow degenerate boundary 
	maps. 
	See also the full definition of a regular cell complex \cite{Curry2013}.
	
	Section~\ref{sec:polyhedral_liftings} explores an alternative approach where the cells 
	are vertically lifted off of the plane to a polyhedron in 3D. This process of lifting a 
	self-stressed graph is widely known as the {\it Maxwell--Cremona correspondence}. The 
	homological relationships governing these lifts are formally expressed in 
	Theorem~\ref{thm:maxwell_lift} using the language of cosheaves.
	
	Section~\ref{sec:nonsphere} delves into the extensions of planar graphic statics to 
	non-spherical topology. This section addresses the construction of force and 
	form diagrams from non-planar graphs, allowing for a broader range of structural 
	configurations. Assurances for the existence of dual diagrams are derived in terms of 
	the genus of the cell structure. In essence, as long as 
	the dimensions of equilibrium stresses exceeds a multiple of the one-dimensional holes 
	present in the topology, the presence of dual diagrams is guaranteed. This insight 
	highlights the role played by topological characteristics of the cell structure.
	
	\begin{theorem*}[\ref{thm:genus}]
		Suppose $(X,p)$ is a regular cellular decomposition of an oriented surface of
		genus $g$ realized in $\R^2$. If the dimension of self stress exceeds $4$ times the 
		genus ($\dim H_1 \calf > 4g$), then there exists a non-trivial parallel realization of 
		$\tilde{X}$ in $\R^2$. The oriented length of each dual	edge is equal to the force 
		within its corresponding primal edge.
	\end{theorem*}
	
	Theorem~\ref{thm:genus} presents a significant extension of 2D 
	graphic statics to encompass non-planar graphs. A similar bound to 
	Theorem~\ref{thm:genus} can be derived for polyhedral lifts. Here, graph edges and 
	faces 
	may overlap in the plane, lifting to a toroidal or higher genus polyhedra.
	
	\begin{theorem*}[\ref{thm:genuslift}]
		Suppose $(X,p)$ is a regular cellular decomposition of an oriented surface of
		genus $g$ realized in $\R^2$. If the dimension of self stress $ \dim H_1 
		\calf$ is greater than $6g$, then there exists a non-trivial polyhedral lift to $\R^3$.
	\end{theorem*}
	
	The existence of non-trivial reciprocal diagrams and polyhedral lifts has previously been 
	linked to zero ``face-edge cycles'' by Crapo and Whiteley \cite{Crapo1993}. The two 
	bounds above simplify and streamline the discovery process.
	%
	\section{Sheaves, Cosheaves, and Structures}
	\label{sec:allsheaves}
	%
	Cellular sheaves and cosheaves are discrete mathematical structures that assign 
	algebraic
	data to the cells of a topological complex. These structures offer a versatile 
	scaffolding for modeling a variety of abstract and physical phenomena. By assigning 
	vector spaces to cells and defining linear maps between adjacent cells, these models 
	allow for the representation of distributed vector-valued data. The pursuit of globally 
	consistent data, where vectors assigned to cells adhere to all local linear constraints, 
	leads to a general concept of global equilibrium. 
	
	For more comprehensive background on sheaves and cosheaves the theses of Curry 
	\cite{Curry2013} and Hansen \cite{hansen2020laplacians} are excellent sources. These 
	provide a deeper background and understanding of the concepts presented here.
	
	\subsection{Cellular Cosheaves}
	\label{sec:cosheaves}
	Cosheaves allocate individualized vector spaces to individual cells, allowing for a 
	localized 
	assignment of data.
	
	\begin{definition}[Cellular Cosheaf] \label{def:cosheaf}
		Given a finite regular  cell complex\footnote{Here all cells are topological disks 
			with boundaries comprised of incident lower dimensional cells. See 
			\cite{Curry2013} for 
			a complete definition.} $X$, a {\it cellular cosheaf} $\calk = \calk_X$ over $X$ 
			consists of 
		the assignment
		\begin{itemize}
			\item to each cell $c \in X$ a finite dimensional vector space $\calk_c$
			with inner product called the {\it stalk} of $\calk$ at $c$ and,
			\item to incident cells $c\lhd d$ in $X$ a linear {\it extension map} $\calk_{d\rhd 
				c}: \calk_d \to \calk_c$.
		\end{itemize}
	\end{definition}
	
	Cosheaves, by nature, map data {\it downward} in cell dimension. 
	Extension maps transfer data from face stalks to edge stalks, and subsequently from 
	edge stalks to vertex stalks. By selecting appropriate bases for these stalks, extension 
	maps can 
	be explicitly represented as matrices. While definition \ref{def:cosheaf} is extremely 
	general, the data representation and behavior of a specific cosheaf can be {\it 
		programmed} by choosing stalks and extension maps.
	
	\begin{definition}[Constant Cosheaf]\label{def:constantcosheaf}
		For $V$ a finite dimensional vector space, the {\it constant
			cosheaf} $\overline{V}$ over a cell complex $X$ is the cosheaf with stalks 
		$\overline{V}_c = V$ over every cell $c$. Extension maps $\overline{V}_{d\rhd c}$ are
		the identity map for each pair of incident cells $c\lhd d$.
	\end{definition}
	
	The constant cosheaf is an elementary cosheaf that assigns ambient vector data to all 
	cells of the topological space $X$. In most cases, an Euclidean space $\R^n$ is selected 
	for $V$, within which other geometric data can be measured and quantified.
	
	With the data of a cosheaf being inherently decentralized, it is necessary to gather and 
	consolidate this data to discern global structure. This total aggregation of stalks across 
	all cells is known as a {\it chain complex}. Within a chain complex, vector data is 
	brought together and organized based on the dimension of cells involved.
	
	\begin{definition}[Chains] \label{def:chains}
		Given a cosheaf $\calk$ over a cell complex $X$, a (Borel--Moore) $i$-{\it 
			chain} is a formal sum $x = \sum x_c$ of terms $x_c \in \calk_c$, where
    $c$ has cell dimension $i$. The {\it boundary} of an $i$-chain $x$ is an $(i-1)$-chain $\bdd x$ that 
		evaluates to
		\begin{equation}
			(\bdd x)_c = \sum_{c \lhd d_i}[ c : d_i ]\calk_{d\rhd c} x_{d_i} 
		\end{equation}
		on an $(i-1)$-dimensional cell $c$.
	\end{definition}
	
	Here, $[\bullet:\bullet] \in \{-1, 0, 1\}$ is a proxy for local orientation of cells called a {\it 
		signed incidence relation}. For an incident pair $c\lhd d$, if 
	the cells' local orientations agree, we set $[c : d]= 1$; otherwise we set $[c : d] = -1$. A 
	singed incidence relation must satisfy the following general rules:
	\begin{itemize}
		\item (Adjacency) $[c:d]\neq 0$ if and only if $c\lhd d$ and $\dim c + 1 = \dim d$.
		\item (Directed Edges) $[u:e][v:e] = -1$ for an edge with incident vertices $u,v\lhd e$.
		\item (Regularity) For any $b\lhd d$, $\sum_{c} [b : c][c : d] = 0$.
	\end{itemize}
  Note that the adjacency condition implies that a signed incidence relation
  encodes the entire structure of a cell complex.
	
	Chains are assembled into a sequence of vector spaces of $i$-chains connected by 
	linear boundary maps.
	
	\begin{definition}[Chain Complex]\label{def:chaincomplex}
		Given a cosheaf $\calk$ over a cell complex $X$, its {\it 
			chain complex} is the sequence of vector spaces of chains and boundary 
		maps given by
		\begin{equation}
			C_i \calk = \bigoplus_{\dim c = i} \calk_c \qquad \qquad \bdd_i: C_i \calk \to C_{i-1} 
			\calk.
		\end{equation}
		The $i$-th boundary map $\bdd_i$ sends an $i$-chain to its 
		boundary $(i-1)$-chain.
	\end{definition}
	
	Chain complexes are fundamental objects in algebraic topology 
	\cite{AlgebraicHatcher2002}, best thought of as a linear-algebraic expansion of the data
	containted within the cosheaf. Such an expansion contains a mixture of essential and 
	redundant
	information. The most efficient compression of this data structure to its core essentials 
	is a classical construction known as {\it homology}.
	
	\begin{definition}[Homology] \label{def:homology}
		Given the chain complex of a cellular cosheaf $\calk$ over a cell complex $X$, its 
		{\it homology} is the sequence of quotient vector spaces $H_i \calk = \ker \bdd_i / 
		\im \bdd_{i+1}$.
	\end{definition}
	
	The {\it kernel} of the boundary map $\bdd_i$ is a subspace of $C_i \calk$. The {\it 
	image} of $\bdd_{i+1}$ 
	is a subspace of $C_i \calk$. Thanks to the regularity of signed incidence relations, it is a 
	fact that $\bdd_{i} 
	\circ \bdd_{i+1} = 0$. Thus, the image of $\bdd_{i+1}$ is also a subspace of the kernel of 
	$\bdd_i$, 
	allowing for a well-defined quotient space $H_i \calk = \ker \bdd_i / \im \bdd_{i+1}$. 
	This $i$-th homology of $\calk$, 
	$H_i\calk$, is therefore a vector space of equivalence classes of {\it $i$-cycles} modulo 
	{\it $i$-boundaries}.
	
	\begin{example}[Classical Cellular Homology]
		\label{ex:classic}
		Consider the case of a graph $X$ with a constant cosheaf $\overline{\R}$ of 
		1-dimensional stalks. 
		The chain complex consists of $C_0\overline{\R}$ and $C_1 \overline{\R}$ with bases 
		the set of vertices and edges respectively. Choosing an (arbitrary) orientation on 
		edges, the 
		resulting boundary map $\bdd: C_1 \overline{\R} \to C_0 \overline{\R}$ is (in more 
		pedestrian
		language) the oriented incidence matrix of the graph. This boundary operator sends a 
		basis edge $e$ to the formal difference of its vertices $\bdd e = u-v$ respecting the 
		orientation of $e$.
		
		The kernel of $\bdd$ is the vector space $H_1 \overline{\R}$ spanned by 
		the oriented cycles in $X$. Here, each 1-cycle explicitly refers to a cycle of the graph. 
		The quotient vector space $H_0 \overline{\R} = C_0 \overline{\R}/\im \bdd$ is simply 
		the {\it 
			cokernel} of $\bdd$; it has basis corresponding to connected components of $X$. 
			So, for 
		example, $\dim H_0\overline{\R}=1$ if and only if the graph $X$ is connected.
		
		This example generalizes to arbitrary regular cell complexes: the homology of the 
		constant cosheaf is the ``classical'' cellular homology of the complex.
	\end{example}
	
	In the same manner that the homology of a simple constant cosheaf captures global 
	topological
	features of the underlying complex -- connectivity, cycles, and more -- a 
	``well-programmed'' cosheaf 
	can encode intricate topological features whose global qualities are revealed by 
	homology.
	
	\subsection{Cellular Sheaves and Duality}
	\label{sec:sheaves}
	In cosheaves, information flows downward in cell dimension by extension maps. The 
	dual data structure is a {\it sheaf}, where information flows up in dimension.
	
	\begin{definition}[Cellular Sheaf]
		Given a cell complex $X$, a {\it cellular sheaf} $\calj = \calj_X$ 
		over $X$ consists of the assignment
		\begin{itemize}
			\item to each cell $c\in X$ a finite dimensional vector space $\calj_c$ with 
			inner product called the {\it stalk} of $\calj$ at $c$ and,
			\item to incident cells $c \lhd d$ in $X$ a linear {\it restriction map} $\calj_{c\lhd 
			d}: 
			\calj_c \to \calj_d$.
		\end{itemize}
	\end{definition}
	Dual to cosheaves, global assignments of data to cells are known as {\it cochains}, 
	which have a {\it coboundary} one dimension higher.
	
	\begin{definition}[Cochains]\label{def:cochains}
		Given a sheaf $\calj$ over a cell complex $X$, an $i$-{\it 
			cochain} is a formal sum of terms $y = \sum y_c$ across all stalks of cells $c$ of 
		dimension $i$, where $y_c\in \calj_c$. The {\it coboundary} of 
		an $i$-cochain $y$ is an $(i+1)$-cochain $\delta y$ that takes value over an 
		$(i+1)$-dimensional cell $d$:
		\begin{equation}
			(\delta y)_d = \sum_{c_i \lhd d}[ c_i : d ]\calj_{c_i\lhd d} y_{c_i} .
		\end{equation}
	\end{definition}
	Here $[\bullet : \bullet]$ is the same signed incidence relation for cosheaves as in 
	Section~\ref{sec:cosheaves}. Cochains assemble into a {\it cochain complex}.
	
	\begin{definition}[Cochain Complex]\label{def:cochaincomplex}
		Given a cosheaf $\calj$ over a cell complex $X$, its {\it cochain complex} is the 
		sequence of 
		vector spaces of cochains with {\it coboundary maps}
		\begin{equation}
			C^i \calj = \bigoplus_{\dim c = i} \calj_c \qquad \qquad \delta^i: C^i \calj \to C^{i+1} 
			\calj.
		\end{equation}
	\end{definition}
	
	\begin{definition}[Cohomology]\label{def:cohomology}
		Given a cellular sheaf $\calj$ over a cell complex $X$, its $i$-th {\it sheaf 
		cohomology} is 
		the quotient space $H^i \calj = \ker \delta^i/ \im \delta^{i-1}$ 
		with cohomology classes consisting of equivalence classes of {\it cocycles}.
	\end{definition}
	
	Converting stalks to their linear dual spaces, sheaves are dual to cosheaves. For a real 
	vector space $V$, let $V^\vee$ denote its dual space of linear functionals $V\to \R$. If 
	$V$ is spanned by column vectors then $V^\vee$ is spanned by row vectors. For every 
	linear map $\phi: V\to W$, its adjoint $\phi^\vee= (-\circ \phi) : W^\vee\to V^\vee$ 
	acts by precomposition with the diagram
	\begin{equation}\label{eq:linprecomp}
		\begin{tikzcd}
			V \ar[rr, "\phi"] \ar[dr, "\psi\circ\phi"'] & & W \ar[dl, "\psi"]\\
			&\R&
		\end{tikzcd}
	\end{equation}
	commuting for $\psi\in W^\vee$. Applied to cosheaves, extension maps are 
	precomposed.
	
	\begin{definition}[Linear Dual Sheaf]
		For a cosheaf $\calk$ define its {\it linear dual sheaf} $\calk^\vee$ by setting
		stalks $(\calk^\vee)_c=(\calk_c)^\vee$. For an incident pair $c\lhd d$ we
		define the restriction map $\calk_{c\lhd d}^\vee = (-\circ \calk_{d\rhd
			c}): \calk_c^\vee \to \calk_d^\vee$ as precomposition by $\calk_{d\rhd c}$, as in 
		Diagram~\ref{eq:linprecomp}.
	\end{definition}
	
	As with all linear maps between finite dimensional vector spaces, there are
	equalities $\ker \bdd^\vee = (\im \bdd)^\perp$ and $\im \bdd^\vee = (\ker
	\bdd)^\perp$. These relations induce isomorphisms on (co)homology.
	
	\begin{theorem}{\cite{Curry2013}} \label{thm:linear_homology}
		Taking linear duals preserves
		(co)homology: $H_i \calk \iso H^i \calk^\vee$ are isomorphic for each $i$.
	\end{theorem}
	
	The second method we will utilize in relating cosheaves and sheaves involves dualizing 
	the 
	underlying cells of $X$ while leaving the data assignments unchanged. If $X$ is an 
	oriented manifold, the Poincar\'{e} dual cell structure $\tilde{X}$ has cell incidence and 
	dimension flipped. Let dual cells have their incidence relations be preserved, so 
	$[\tilde{d} : \tilde{c}] = [c : d]$. With top dimensional cells of $X$ having all
	the same orientation, this relation on $\tilde{X}$ satisfies the requirements of a signed 
	incidence relation.
	
	\begin{definition}[Poincar\'e Duality]\label{def:poincare}
		Suppose $\calk = \calk_X$ is a cosheaf over an oriented $n$-manifold $X$. Define its
		{\it Poincar\'{e} dual sheaf} $\tilde{\calk} = \tilde{\calk}_{\tilde{X}}$ over the dual cell 
		complex $\tilde{X}$ by taking identical stalks $\tilde{\calk}_{\tilde{c}} = \calk_c$ and 
		maps $\tilde{\calk}_{\tilde{d}\lhd \tilde{c}} = \calk_{d\rhd c}$ for $c\lhd d$ in $X$.
	\end{definition}
	
	Both the cosheaf $\calf$ over $X$ and the sheaf $\tilde{\calf}$ over
	$\tilde{X}$ encode the same information as their stalks and
	extension/restriction maps are identical. The difference is their domains;
	$\calf$ is a cosheaf over $X$ and $\tilde{\calf}$ is a sheaf over $\tilde{X}$.
	
	Changing the cell structure at the level of cochains formally leads to the
	isomorphism $C_k \calf \iso C^{n-k}\tilde{\calf}$ for
	$0\leq k\leq n$. The boundary $\bdd$ and coboundary $\tilde{\bdd}$ each
	perform the same operation on dual elements, leading to the isomorphism $H_k \calf 
	\iso 
	H^{n-k} \tilde{\calf}$ in (co)homology known as {\it Poincar\'e Duality}.\footnote{We 
		note that if $X$ is closed, then $\tilde{X}$ is open and vice-versa. This matters in 
		particular when $X$ is a manifold with boundary. In this case, Poincar\'e duality is 
		called {\it Lefschetz duality} and links homology with relative cohomology 
		\cite{AlgebraicHatcher2002}.}
	
	\begin{example}
		Suppose $\overline{V}$ is a constant cosheaf over $X$. Both the linear dual 
		$\overline{V}^\vee$ and the Poincar\'e dual $\tilde{\overline{V}}$ are sheaves, but 
		these are over different spaces $X$ and $\tilde{X}$. Combining these operations, 
		$\tilde{\overline{V}}^\vee$ is a cosheaf over $\tilde{X}$ and is what we will call the 
		{\it 
			reciprocal cosheaf} to $\overline{V}$. All of these structures hold the same 
			information 
		(with isomorphic (co)homology).
	\end{example}
	
	\subsection{Laplacians and Diffusion}
	\label{sec:sheaflaplacian}
	One final ingredient remains which connects sheaves and cosheaves, cohomology and 
	homology, with dynamics. 
	This is a generalization of the graph Laplacian to the setting of cellular sheaves, known 
	as a {\it sheaf
		Laplacian} \cite{Hansen2019}. Given a cellular sheaf with coboundary map $\delta$, 
		define the 
	Laplacian as 
	\begin{equation}
		\label{eq:sheaflaplacian}
		\Delta = (\delta + \delta^\vee)^2 = \delta^\vee\circ\delta + \delta\circ\delta^\vee .
	\end{equation}
	When restricted to 0-cochains, this Laplacian takes the simpler form of 
	$\delta^\vee\circ\delta$, 
	pushing data from vertices out to incident edges, then pulling back and incorporating
	data pushed 
	from neighboring vertices out to the edges. 
	
	Using the sheaf Laplacian on 0-cochains as a diffusion operator leads to a heat equation 
	on cochains
	with well-behaved properties. The following results from The following results are 
	relevant. 
	\begin{theorem}[\cite{Hansen2019,hansen2021opinion}]
		\label{thm:sheafdiffusion}
		Given a cellular sheaf of finite-dimensional inner product spaces and $\xi_0\in C^0$ a 
		0-cochain, 
		the differential equation
		\begin{equation}
			\label{eq:heat}
			\frac{d}{dt}\xi = -\alpha\Delta\xi \quad ; \quad \alpha>0,
		\end{equation}
		with $\xi_0$ as initial condition has solutions which converge to the nearest 
		cohomology class 
		$\xi_\infty\in H^0 = \ker\Delta$. 
	\end{theorem}
	One can likewise define Laplacians which act on higher-dimensional cochains, or, via 
	dualizing, which act on cosheaves.
	We will use the sheaf Laplacian in Sections \ref{sec:stiffness} and \ref{sec:position}  to 
	address issues of stiffness and more. 
	
	\section{Sheaves and Cosheaves for Linkages}
	\label{sec:constructions}
	
	This section introduces several novel cellular sheaves and cosheaves for linkages
	modelled as cell complexes realized in a Euclidean space.
	A {\it realization} of an abstract cell complex $X$ is a map $p$ from the vertex set of $X$ to
 	$\R^n$ that assigns to each vertex explicit coordinates. 
	We will generally require coordinates to be unique so that 
	edges are realized as nonsingular lines. We will call the combined system $(X,p)$ a {\it 
	diagram}, the term most used in the literature surrounding graphic statics 
	\cite{MechanismsMitchell2016} (alternatively Crapo and Whiteley use the term 
	``framework'' in their influential work \cite{Crapo1993} --- further used in rigidity 
	theory \cite{FrameworksConnelly2015}).
	
	\subsection{The Force Cosheaf} 
	\label{sec:force}
	
	We begin our constructions with the force cosheaf, denoted $\calf$ (see 
	Figure~\ref{fig:force_sketch}). This cosheaf provides an accurate description 
	of the static loads and mechanisms of a pin-jointed truss. The underlying 
	geometry of the base linkage $X$ is realized in a Euclidean space $\R^n$.
	
	\begin{figure}[ht]
		\centering \includegraphics[scale=0.6]{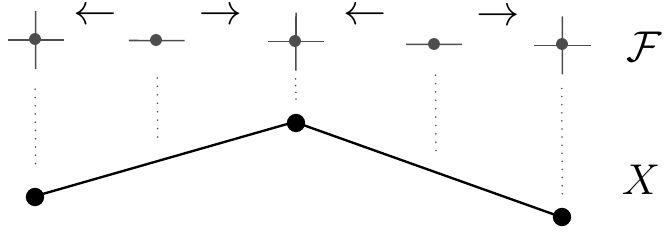}
		\caption{An abstract sketch of the force cosheaf $\calf$ over $\R^2$ is pictured. The 
		stalks of $\calf$	are situated ``above'' the cells of $X$ with linear maps respecting	
		incidence.}
		\label{fig:force_sketch}
	\end{figure}
	
	Although this initial paper focuses on the planar setting ($n=2$), we present the
	definitions in more generality to permit higher-dimensional graphic statics in future 
	work. 
	
\begin{definition}[Force Cosheaf]
\label{def:forcecosheaf}
	Let $(X,p)$ be a diagram with injective realization. The {\it force cosheaf} $\calf$ 
	over $(X, p)$ has stalks that keep track of internal and external forces as follows. 
	Each vertex $v$ has stalk $\calf_v = \R^n$ thought of as 
	a tangent space at $p(v)$ with unit basis vectors $\partial_i=\partial/\partial x_i$. 
	Each edge $e$ has stalk $\calf_e=\R$ thought of as a 1-D space tangent to the edge 
	with basis vector $\partial_t$. For an edge $e$ connecting vertices $u$ and $v$, the map 
	$\calf_{e\rhd v}$ sends the unit basis vector $\partial_t$ to the unit normalization
	of the vector $p(v)-p(u)$. All stalks over higher-dimensional cells of $X$ vanish.
\end{definition}

\begin{remark}
	One can profitably think of the geometric realization of each edge $e$ with endpoints 
	$u\lhd e$ and $v\lhd e$ as the straight interval passing through the points $p_v$ and $p_u$. 
 	From this perspective, the stalk $\calf_e$ can be thought of as the 1-D vector subspace ${\bf e} = 
	\text{span}\{p_v - p_e\}$ of $\R^n$ parallel to the realization of $e$ induced by $p$. 
	For an incident vertex $v$, the map $\calf_{e \rhd v}$ is the inclusion of this subspace 
	${\bf e}$ into $\R^n = \calf_v$, {\em cf.} the inclusion of the tangent space of a submanifold 
	into the ambient manifold. See Example \ref{ex:force_embed} for an instance where this perspective 
	is useful.
\end{remark}

The stalks represent forces applied at the joints (vertex stalks) and axial forces or stresses
along the linkage element (edge stalks). The extension maps record how axial forces are 
distributed to the pin joints. Since $C_i\calf=0$ for $i > 1$, there are at most two non-zero 
homologies, determined by the boundary map $\bdd : C_1 \calf \to C_0 \calf$ : the kernel of $\bdd$, 
$H_1 \calf$, and the cokernel of $\bdd$, $H_0 \calf$. The chains and homologies of $\calf$ have 
the following useful interpretations in structural mechanics.
	
	\begin{itemize}
		\item Each 0-chain in $C_0 \calf$ is a distribution of {\it static loads}
		applied to vertices, i.e., a force vector in $\R^n$ applied to each vertex. 
		\item Each 1-chain in $C_1 \calf$ is a distribution of internal {\it axial forces}, or {\it 
			stresses} over the edges. The sign of this stress value in combination with the local 
		orientation of the edge determines whether the stress is a tension (positive) or 
		compression (negative). 
		\item The boundary map $\bdd : C_1 \calf \to C_0 \calf$ distributes 
		the forces resulting from internal stresses over edges to adjacent vertices, summing
		at each vertex.
		\item The 1st homology $H_1 \calf = \ker \bdd$ is the space of {\it axial self
			stresses} or {\it equilibrium stresses} of the truss. Each self stress satisfies an 
		equilibrium equation over vertices, resulting in a net-zero force everywhere. For 
		example, Figure~\ref{fig:force_boxed} portrays a self-stressed truss.
		\item The 0th homology $H_0 \calf = \coker\ \bdd$ is the space of {\it constrained
			degrees of freedom} of vertices. These are assignments of forces to
		vertices that cannot result from tensile or compressive force within edges.  
		These vertex forces then must instead impart {\it infinitesimal accelerations} to the
		structure: rotations, translations, or mechanisms: see Section~\ref{sec:linkage}.
	\end{itemize}
	
	We emphasize that struture and homology of $\calf$ are highly dependent on the geometry of the 
	diagram $(X,p)$. Varying the realization $p$ will change the force cosheaf,
  even if the topology of $X$ does not change.
 	This is why, in the context of graphic statics, $(X,p)$ -- the {\it form diagram} -- is the
	crucial object of study. The force cosheaf is a richer algebraic structure
  that reveals the deeper implications of the form diagram.
	
	\begin{example}[Rigidity]\label{ex:rigidity}
		The zeroth homology detects rigidity. A diagram $(X,p)$ in $\R^n$ 
		is {\it rigid} if there are no {\it mechanisms} -- that is, if the only nonzero infinitesimal
		accelerations generate the special Euclidean group of rotations and translations. By 
		the above interpretation of the force cosheaf homology, the diagram is rigid if and only if 
		the dimension of $H_0 \calf$ is $\binom{n+1}{2}$, the dimension of the 
		special Euclidean group of global translations and rotations in $n$ dimensional space. 
		For example, an $n$-simplex realized in $\R^n$ with generic vertex coordinates is 
		rigid as well as Figure~\ref{fig:force_boxed}. For non-rigid graphs, other generators of 
		$H_0 \calf$ correspond to {\it mechanisms}.
	\end{example}
	
	\begin{figure}[ht]\centering
		\includegraphics{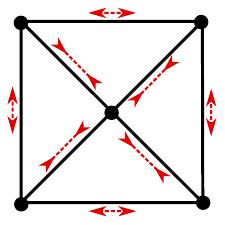}
		\caption{A diagram $(X,p)$ is pictured where $\dim C_0 \calf = 10$ and $\dim C_1 
			\calf = 8$. The boundary map is not injective, with kernel 
			spanned by the cycle that takes values $1$ on the four outside edges (tension) and 
			takes values $-\sqrt{2}$ on the four interior edges (compression). 
			This 1-cycle is a generator of $H_1 \calf \cong \R$.}
		\label{fig:force_boxed}
	\end{figure}

	\subsection{The Linkage Sheaf} 
	\label{sec:linkage}
	
	The linear dual of the force cosheaf is what we will call the {\it linkage sheaf} modeling 
	the kinematics of a bar-and-joint linkage. The {\it principle of virtual work} implies that 
	displacements are dual to forces and that the 
	displacement coboundary map $\delta$ is the adjoint of the force boundary map 
	$\bdd^\vee$ \cite{Calladine1978}.
	
	\begin{definition}[Linkage Sheaf]
	\label{def:linkagesheaf}
		The {\it linkage sheaf} $\calf^\vee$ is the linear dual of the force cosheaf $\calf$. 
		Again, fixing a diagram $(X,p)$, vertex stalks $\calf^\vee_v$ are dual spaces 
		$(\R^n)^\vee$ with basis 1-forms $dx_i$; the edge stalks are one-dimensional dual 
		spaces $\calf^\vee_e \iso \R^\vee$ with basis $dt$, and all stalks over higher 
		dimensional cells vanish. For each edge $e$ between $u$ and $v$, the restriction 
		map $\calf^\vee_{v\lhd e}:(\R^n)^\vee \to {\R}^\vee$ is the
		linear dual of the map $\calf_{v\lhd e}:\R\to\R^n$, sending a covector $\alpha\in(\R^n)^\vee$ at $v$ 
  		to the multiple of $dt$ with coefficient $\alpha(p_e)$ where $p_e$ is the unit 
  		normalization of the vector $[u:e](p_u - p_v)$.
	\end{definition}

If one thinks of the force cosheaf as having a map $\calf_{e \rhd v}$ given by the inclusion 
of the subspace ${\bf e}$ ``parallel'' to the geometric edge into $\R^n$, then the linkage 
sheaf map $\calf_{v\lhd e}:(\R^n)^\vee\to{\bf e}^\vee$ is its linear dual.

 	\begin{itemize}
		\item The space of 0-cochains $C^0 \calf^\vee$ consists of {\it infinitesimal 
		displacements} of 
		vertices in $\R^n$. 
		\item The space of 1-cochains $C^1 \calf^\vee$ is the space of {\it axial extensions 
		and contractions}. Normalizing by the length of each edge gives the {\it strain}. 
		\item The coboundary map $\bdd^\vee: C^0 \calf^\vee \to C^1 \calf^\vee$ computes 
		the infinitesimal or 1st-order changes in the lengths of edges resulting from a given displacement of 
		the truss vertices.
		\item The 0th cohomology $\ker\bdd^\vee = H^0 \calf^\vee$ consists of {\it 
			constrained stiff displacements} assigned to vertices that preserve edge length.
		Each constrained displacement rotates or translates the linkage while those edges 
		remain rigid. 
		\item The 1st cohomology $H^1 \calf^\vee$ consists of equivalence classes of {\it
			unrealizable axial deformations} --- those axial deformations that cannot result 
			from movements of vertices. These can be thought of as strain measurement reading errors.
	\end{itemize}
	
	As these are cotangent spaces (duals to tangent spaces), the displacements to the 
	vertices should be thought of as either infinitesimal or (using an exponential map) 
	1st-order terms 
	in the expansion of the displacement. By Theorem \ref{thm:linear_homology}, the 
	cohomology of the linkage sheaf is isomorphic to the homology of the force cosheaf. 
 	This provides a useful way of re-interpreting each (co)homology.
	
	\begin{example}[Impossible Axial Deformation]
		For the diagram pictured in Figure~\ref{fig:force_boxed}, it is impossible for the four 
		interior edges to grow in length while the four outside edges shrink. This cocycle is a 
		generator for $H^1 \calf^\vee$ and is dual to the self stress. This coincides with 
		the isomorphism $H^1 \calf^\vee \iso H_1 \calf$ following 
		Theorem~\ref{thm:linear_homology}.
	\end{example}
	
	\subsection{Maxwell's Rule and Euler Characteristic}
	\label{sec:rule}
	
	Additional tools from algebraic topology enable elegant reformulations of principles in 
	static analysis. One such tool is the classical {\it Euler characteristic}, adapted 
	to cellular cosheaves. 
	
	Recall the elementary topological fact that for a finite polyhedral spherical surface $X$ 
	with $|V|$ vertices, $|E|$ edges, and $|F|$ faces, the Euler characteristic 
	$\chi\left(X\right)=|V|-|E|+|F|$ equals $+2$, independent of how the spherical surface is 
	discretized. The reason for this is an elementary but fundamental result. For a finite 
	sequence of finite-dimensional vector spaces $C=(C_i)$, its Euler characteristic 
	$\chi(C)$ is the alternating sum of dimensions of the vector spaces. In the context of 
	a chain complex, there is an important and fundamental relationship to homology:
	\begin{lemma}{\cite{AlgebraicHatcher2002}}
	\label{lem:euler}
		The Euler characteristic of a chain complex $C = \left( C_i \right)$ and its homology 
		$H = \left( H_i \right)$ agree:
		\begin{equation}\label{eq:euler}
			\sum_{i}{\left(-1\right)^i \dim C_i} 
			= 
			\chi(C)
			=
			\chi(H)
			=
			\sum_{i}{\left(-1\right)^i \dim H_i},
		\end{equation}
	\end{lemma}
	This is why for a triangulated spherical surface $\chi=2$, since the classical homology of 
	a cell complex comes from the constant cosheaf (Example \ref{ex:classic}) and this 
	cosheaf has Euler
	characteristic given by $|V|-|E|+|F|$, while the homology of a sphere is 
	$\dim H_0 = \dim H_2 = 1$ and $\dim H_k = 0$ for all other $k$. Lemma~\ref{lem:euler} also
 	holds for the cohomology of a cochain complex (by simple duality).
	
	When applied to the force cosheaf (or the linkage sheaf), one immediately generalizes Example 
	\ref{ex:rigidity} to obtain the classic {\it Maxwell's Rule} \cite{ReciprocalMaxwell1864}. 
	The modern form for two dimensional trusses is given in \cite{Calladine1978} as
	\begin{equation}\label{eq:maxwell_rule}
		2|V| - |E| = 3 + |M| - |S| ,
	\end{equation}
	where $|M|$ is the number of linkage mechanisms, $|S|$ is the number of 
	self-stresses, 
	$|V|$ is the number of vertices of the truss, and $|E|$ is the number of its edges. The 
	prevalence of differences suggests that this is related to Euler characteristic. 
	
	\begin{theorem}[Maxwell's Rule in dimension $n$]
		\label{thm:maxwell}
		Given a diagram $(X,p)$ in $\R^n$, the generalized Maxwell's Rule holds:
		\begin{equation}\label{eq:euler_force}
			n|V| - |E| = \binom{n+1}{2} + |M|- |S|
		\end{equation}
	\end{theorem}
	\begin{proof}
		Apply Lemma~\ref{lem:euler} to the force cosheaf, noting that in dimension $n$, 
		the dimension of $H_0\calf$ equals the sum of the number of the Euclidean motions 
		($n(n+1)/2$) plus the number of linkage mechanisms $|M|$.
	\end{proof}
	
	Dualization implies that the theorem holds as well for the linkage sheaf, which makes
	interpreting the linkage mechanisms clearer and gives an equivalence between the 
	number of self-stresses and the number of unrealizable axial deformations.
	
	\subsection{Stiffness and the Sheaf Laplacian}
	\label{sec:stiffness}
	
	The relationship between stress and strain is encoded in {\it Hooke's Law} for ideal 
	elastic elements. Under deformation, the strain within an axial element is proportional 
	to its force output	by its spring constant $\kappa_e$, parameterized by the
	edges $e$ of the diagram $X$. Multiplying by $\kappa_e$ determines an explicit
  isomorphism between the (one-dimensional) vector space of axial extensions and
  contractions of $e$ (i.e. $\calf^\vee_e$) and the vector space of stresses of
  the element ($\calf_e$). The quadratic form of
  this inner product $\langle a, b \rangle_e = \kappa_e ab$ on $\calf_e^\vee$ represents 
  the work done by a 
  given axial deformation of the element. Using this inner product
  rather than the obvious one induced by the ambient space is useful as a way of
  determining weights for a sheaf Laplacian on the linkage sheaf $\calf^\vee$
  (recall Section~\ref{sec:sheaflaplacian}). We begin with a simple classical example.
	
	\begin{example}
	\label{ex:stiffness}
		For a single edge $e$ in $\R^2$ between vertices at points $p(u)$ and $p(v)$, the 
		component {\it stiffness matrix} $K_e$ for this element is typically defined
    by an equation like
		\begin{equation}\label{eq:local_stiffness}
		K_e 
		=
		\kappa_e
		\begin{small}\begin{bmatrix}
			\cos^2(\theta) & \cos(\theta)\sin(\theta) & -\cos^2(\theta) & 
			-\cos(\theta)\sin(\theta)\\
			\cos(\theta)\sin(\theta) & \sin^2(\theta) & -\cos(\theta)\sin(\theta) & 
			-\sin^2(\theta)\\
			-\cos^2(\theta) & -\cos(\theta)\sin(\theta) & \cos^2(\theta) & 
			\cos(\theta)\sin(\theta)\\
			-\cos(\theta)\sin(\theta) & -\sin^2(\theta) & \cos(\theta)\sin(\theta) & 
			\sin^2(\theta)\\
		\end{bmatrix} \end{small},
		\end{equation}
		where $\theta$ is the angle at which the edge $e$ is inclined in the Euclidean plane. 
		The {\it total stiffness matrix} is the block matrix $K$ comprised of sums of matrices 
		$K_e$ over every edge.
	\end{example}
	
	The stiffness matrix is a symmetric linear operator that takes a vector of infinitesimal 
	displacements of vertices to a vector of forces at the vertices.
  Another perspective is that, as a quadratic form, the stiffness matrix
  represents the infinitesimal work done by an infinitesimal displacement of the
  vertices, or, equivalently, that $\langle \xi, K \xi \rangle$ is (twice) the potential
  energy stored in the truss as a result of the deformation $\xi$.

  We claim that the stiffness matrix is really the sheaf Laplacian of
  $\calf^\vee$ with respect to the inner product given by the spring constants
  $\kappa_e$. This is simple to see from the perspective of the quadratic form,
  beginning with the observation that $\langle  \xi, \Delta \xi \rangle = \langle  \delta \xi, 
  \delta \xi
  \rangle$. Since the coboundary $\delta$ computes the
  infinitesimal edge deformations given by vertex displacements, and the inner
  product computes (and sums) the work done by each edge deformation, this is
  precisely the work done on the truss by the vertex displacement $\xi$. One can
  also show this fact by computing the matrix entries of $\Delta$ and showing
  that they are equal to those of $K_e$ as defined above~\eqref{eq:local_stiffness}, but 
  this is more tedious and less enlightening than the work-based proof.
	
	This interpretation of the stiffness matrix as a weighted sheaf Laplacian allows for the
	following result, an immediate consequence of Theorem~\ref{thm:sheafdiffusion}.
	
	\begin{corollary}
		\label{cor:heatstiffness}
		For $\xi_0\in C^0\calf^\vee$ an initial condition of vertex displacements of a diagram 
		$(X,p)$, the diffusion equation
		\begin{equation}\label{eq:spring_diffeq}
			\frac{d\xi}{dt} = -\Delta_{\calf^\vee} \xi
		\end{equation}
		with $\Delta_{\calf^\vee}$ the sheaf Laplacian of $\calf^\vee$, converges to a vertex 
		displacement 
		$\xi_\infty \in H^0\calf^\vee$ representing a stiff displacement for which all 
		members have zero 
		internal force.
	\end{corollary}

Utilizing mass and dampening matrices as well as the stiffness matrix, the wave equation 
is another legitimate model for truss dynamics.

	
	\subsection{The Position Sheaf} 
	\label{sec:position}

	In Section~\ref{sec:linkage} we saw how $H^0$ of the linkage sheaf 
	$\calf^\vee$ determines constrained stiff displacements of the truss. Under such an 
	action, the truss is manipulated by a mechanism or global motion without extending 
	edges. Here we consider the opposite: the space of deformations on a truss that {\it 
	only} extend and contract edges, keeping them parallel. This space of parallel motions is 
	captured by the cohomology of the following sheaf.
	
	\begin{definition}[Position Sheaf]
	\label{def:positionsheaf}
		Suppose $(X,p)$ is a diagram in Euclidean space $\R^n$ with injective realization. The 
		{\it position sheaf} $\calj$ is built as follows. Consider the sheaf whose vertex stalks 
		are the dual spaces $(\R^n)^\vee$. Assign to each edge $e$ the stalk the quotient of 
		$(\R^n)^\vee$ by the subspace $\calf^\vee_e = {\bf e}^\vee$ from 
		Definition~\ref{def:linkagesheaf}. All higher dimensional cells have zero 
		stalks. The restriction maps $\calj_{v\lhd e}$ are precisely the projections onto these 
		quotient spaces.
	\end{definition}
	
	\begin{figure}[ht]\centering
		\begin{subfigure}[t]{0.49\textwidth}\centering
			\includegraphics{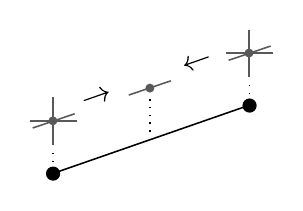}
			\caption{Sketch of the linkage sheaf in $\R^3$}
		\end{subfigure}
		\begin{subfigure}[t]{0.49\textwidth}\centering
			\includegraphics{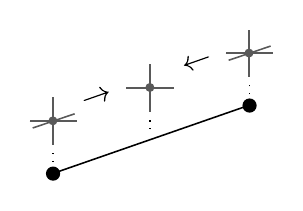}
			\caption{Sketch of the position sheaf in $\R^3$}
		\end{subfigure}
	\end{figure}
	
	The cochain complex and the cohomology of $\calj$ can be interpreted as follows.
	\begin{itemize}
		\item The 0-cochains $C^0 \calj$ are infinitesimal deformations, as per the linkage 
		sheaf.
  		\item The 1-cochains $C^1 \calj$ are systems of {\it orthogonal deviations} of edges 
		with respect to the original realization $p$. Over an edge $u,v\lhd 
		e$ in $(X,p)$, an element $o_e \in \calj_e$ is a co-vector	perpendicular (or quotient) 
		to the edge. If one fixes a reference frame at one vertex $u$ and ``looks down the 
		edge'' in $\R^n$, we let the co-vector $[u:e] o_e$ act on vertex $v$ by displacing 
		it in the plane of view. No displacement of $v$ out of plane (or parallel to $e$) 
		can be detected however. These 1-chains can be considered as 1-forms in the normal 
		subspace to the edge. For small values of $o_e$, the small angle approximation $o_e 
		= \sin(\theta)\sim \theta$ means we consider $o_e$ as a {\it small rotation} of the 
		edge by $\theta$. This is pictured in Figure~\ref{fig:position_square} (right).
		
		\item The coboundary $\delta: C^0 \calj \to C^1 \calj$ takes a deformation
		of $(X,p)$ and measures its induced orthogonal deviation over edges. Specifically, 
		given $\xi\in C^0\calj$, the coboundary $\delta \xi$ evaluated at edge $e$
		describes the net orthogonal motion of $e$ under the realization $p+\xi$.  
		
		\item The 0-cohomology $H^0 \calj$ consists of {\it parallel extensions} of $X$,  
		deformations to vertices such that edges remain parallel to their original. For an 
		orthogonal deviation $\xi\in H^0 \calj$, over an edge $u,v\lhd e$ and at a vector 
		$x_e\in {\bf e}^\perp$ the co-vector $\xi_v-\xi_u$ takes value $(\xi_v-\xi_u)(x_e) = 
		0$: thus, $\xi_v=\xi_u$, and the induced edges in the deformed 
		diagram $(X,p+\xi)$ are parallel to those in $(X, p)$.
 
		\item The cohomology $H^1 \calj$ is the space of equivalence classes of {\it 
		impossible orthogonal deviations} or {\it impossible edge rotations}. These are error 
		assignments to edges which cannot result from any possible choice of vertex 
		coordinates. Figure~\ref{fig:position_boxed} depicts a prototypical cocycle.
	\end{itemize}
  %
  %
 	Parallel deformations in $H^0 \calj$ are co-vectors and can be summed, resulting in a 
 	visual algebra of shape. For instance, if $\xi_0, \xi_1\in H^0 \calj$ are two cocycles, the 
 	deformations can be added to the diagram $(X, p+\xi_0 + \xi_1)$. Moreover the 
 	realization $p$ (differentially $dp$) itself is a parallel deformation, {\it scaling} with 
 	respect to the origin $(X,p)$ to the diagram $(X, p + dp)$. To better understand the 
 	vector space we can observe the {\it realization of infinitesimals} $(X, \xi)$ for $\xi\in 
 	H^0 \calj$. There is an isomorphism between diagrams of the form $(X, \xi)$ and 
 	diagrams of the form $(X, p + \xi$). This is justified by canonical isomorphisms between 
 	tangent spaces $T_{p_v} \R^n$ and the base manifold $\R^n$. 
 	Figure~\ref{fig:position_square} (left) and (center) depicts this duality between 
 	infinitesimal rotations and scalings.
 		
	\begin{figure}[ht]\centering
		\begin{subfigure}[t]{0.3\textwidth}\centering
			\includegraphics{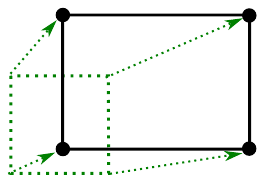}
		\end{subfigure}
		\begin{subfigure}[t]{0.3\textwidth}\centering
			\includegraphics{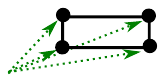}
		\end{subfigure}
		\begin{subfigure}[t]{0.2\textwidth}\centering
			\includegraphics{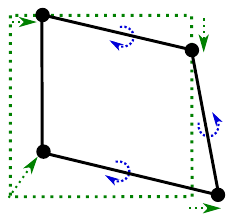}
		\end{subfigure}
		\caption{In an example with $(X,p)$ a square, the position sheaf cochain complex 
			has $C^0 \calj$ of dimension eight and $C^1 \calj$ of dimension four with a surjective 
			coboundary map. Two generators of parallel realizations in $H^0 \calj$ translate 
			the square in space while the two other generators change the aspect ratio of the 
			square, fully describing all rectangles in the plane (left). Where a cocycle $r\in H^0 
			\calj$ deforms the diagram $(X, p)$ to $(X, p+r)$ (left), one can instead depict $r$ 
			as a diagram of infinitesimals (center), as a realization of $X$ in its own right $(X, 
			r)$. In $C^1 \calj$, 
	  		each 1-cochain measures the edge rotation of an alternative diagram relative to original $(X, p)$ (right). 
   			For the square, $H^1 \calj = 0$, meaning every possible orthogonal edge deviation in $C^1 \calj$ is the 
			coboundary of some choice of vertex coordinates in $C^0 \calj$. A cochain is drawn 
			(right), with orthogonal edge deformations drawn more perceptibly as (small) rotations.}
		\label{fig:position_square}
	\end{figure}
 
	\begin{example}[Impossible Edge Rotation] \label{ex:position_boxed} 
		Suppose $(X, p)$ is the form diagram in Figure~\ref{fig:force_boxed}. In constructing alternative 
		parallel realizations, notice the diagonal edges lock the aspect ratio of 
		the outer square. Consequently the vector space $H^0 \calj$ consists of two 
		translations and one scaling dimension. The Euler characteristic 
		formula~\eqref{eq:euler} for sheaves indicates that $H^1 \calj$ is one dimensional, with 
		a generator pictured in Figure~\ref{fig:position_boxed} (left). Just as mechanisms of a 
		truss cannot be resisted by truss members (inducing motion), impossible edge rotations 
		cannot be realized by nodal positions.
	\end{example}
	
	\begin{figure}[ht]\centering
		\begin{subfigure}[t]{0.49\textwidth}\centering
			\includegraphics{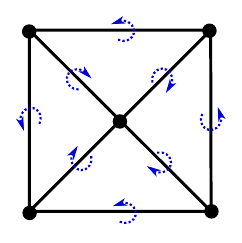}
		\end{subfigure}
		\begin{subfigure}[t]{0.49\textwidth}\centering
			\includegraphics{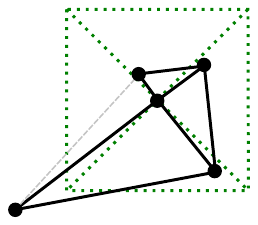}
		\end{subfigure}
		\caption{An impossible infinitesimal edge rotation, illustrating an element of $H^1 \calj$, 
  			is sketched (left). To see this as an impossible class, first fix the center vertex. 
			The four diagonal edges act by rotating the four corner vertices clockwise. No matter 
			how one contracts or extends the diagonal edges, at least one exterior edge must rotate 
			clockwise as well (right). Specifying the outside edges to all rotate counterclockwise 
			is a contradiction -- an impossible requirement.}
		\label{fig:position_boxed}
	\end{figure}
	
	\begin{example}[Self Shear]\label{ex:self shear}
		The position sheaf has its own linear dual $\calj^\vee$. For $\calj$ a position sheaf over 
		the diagram $(X,p)$, the cosheaf $\calj^\vee$ has edge stalks ${\bf e}^\perp \iso \R^{n-1}$ 
  		over edges, with extension maps embeddings $\calj^\vee = E_{e^\perp}$ 
		from these perpendicular subspace into $\R^n$ vertex stalks.
		
		Similar to the force cosheaf, the homology $H_1 \calj^\vee$ has interpretation as a 
		space of {\it self-shear stresses}. 
		Here, the values of a stalk $\calj^\vee _e \iso {\bf e}^\perp$ is the space of shear 
  		forces perpendicular to the member $e$. This 
		equilibrium of these forces is {\it self-shear}, where no edge is loaded with any axial 
		force.
	\end{example}
	
	\begin{example}[Position Sheaf Laplacian]
	\label{ex:pos-laplacian}		
		Akin to the stiffness matrix in Section \ref{sec:stiffness}, 
		the sheaf Laplacian of the position sheaf $\Delta_\calj$ 
		has interesting properties. Although spring constants can be introduced, consider the
  		unweighted sheaf Laplacian $\Delta_\calj = \delta \circ \delta$. For the 
		position sheaf $\calj$ over a reference diagram 
		$(X,p)$, the diagram $(X,p+\xi_t)$ with arbitrary initial realization $\xi_0\in C^0 \calj$ 
		undergoes transformation under the differential equation
		\begin{equation}\label{eq:pos_diffeq}
			\frac{d\xi}{dt} = -\Delta_\calj\xi.
		\end{equation}
		The realization displacements $\xi_t$ converge to $\xi_\infty \in H^0 \calj$,
		a parallel realization to $p$. The limit diagram $(X,p+\xi_\infty)$ is the closest 
		diagram to $(X,p+\xi_0)$ parallel to $(X,p)$ in the sense that $\xi_\infty$ minimizes 
		the expression \cite{hansen2020laplacians}
		\begin{equation}
			\|\xi_\infty - \xi_0 \| = \sqrt{\sum_{v\in V} \left\|\xi_\infty(v) - \xi_0(v)\right\|^2}.
		\end{equation}
		
		\begin{figure}[ht]\centering
			\includegraphics{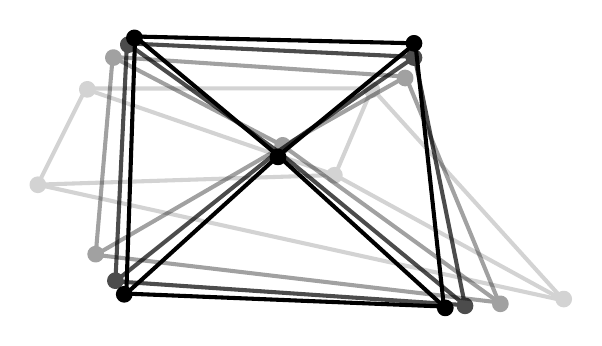}
			\label{fig:pos_laplacian}
			\caption{The action of the position sheaf Laplacian (derived from 
				Figure~\ref{fig:force_boxed}) is pictured on a diagram with arbitrary 
				realization $(X,p+\xi_0)$. Under the heat equation~\eqref{eq:pos_diffeq}, the 
				realization exponentially converges to that of the limit diagram $(X,p+\xi_\infty)$ 
				parallel to that of Figure~\ref{fig:force_boxed}.}
		\end{figure}
	\end{example}

	\begin{example}[Maxwell Dualized]
	\label{ex:maxwell-dualized}		
	Recall from Section \ref{sec:rule} that the classical Maxwell Rule follows from the 
	invariance of the Euler characteristic under (co)homology of the force cosheaf or the 
	linkage sheaf. The same can be directly applied to the position sheaf and its dual 
	cosheaf, with corresponding interpretations. 
  	We leave this variation of the Maxwell Rule to the curious reader as an illustrative exercise.
	\end{example}
 
	\subsection{Maps Between Cosheaves}
	\label{sec:maps}
	Any given cosheaf is best understood in relation to other cosheaves. 
	These connections are algebraically formed by distributed linear maps between stalks, 
	subject to constraints ensuring the consistency of the transmitted data.
	
	\begin{definition}[Cosheaf Map]
		For two cosheaves $\calk$ and $\call$ over the same cell complex $X$, a {\it cosheaf map} 
		$\phi: \calk \to \call$ consists of a collection of linear maps $\{\phi_c: \calk_c \to 
		\call_c\}$ on stalks satisfying, for each pair of adjacent cells $c\lhd d$, the following
		commutativity condition:
		\begin{equation}\label{eq:comm_square_cosheaf}
			\begin{tikzcd}
				\calk_d \ar[r, "\phi_d"] \ar[d, "\calk_{d\rhd c}"] & \call_d \ar[d, "\call_{d\rhd 
					c}"]\\
				\calk_c \ar[r, "\phi_c"] & \call_c
			\end{tikzcd}
		\end{equation}
		Commutativity here means the compositions of maps $\calk_d\to\call_c$ are 
		path-independent.
	\end{definition}
	
	Each cosheaf map $\phi: \calk \to \call$ induces a map between chain complexes $\phi:
	C \calk \to C \call$ by combining all stalk-wise linear component maps. A cosheaf
	morphism $\phi$ is injective, surjective, or an isomorphism if each of its
	component maps $\phi_c$ are.
	
	Intuition from linear algebra holds: an injective cosheaf map $\phi: \calk \to \call$ 
 	can be thought of as a ``sub-cosheaf'' -- an inclusion of one data structure into another. 
  	To model the data in $\call$ that is ``orthogonal'' to $\calk$, 
	one builds the {\it quotient cosheaf} $\calk / \call$ with stalks being quotient vector 
	spaces $\left( \call / \calk \right)_c = \call_c / \calk_c$ 
	with $c$ ranging over all cells. Extension maps over incident cells $c\lhd d$ are derived 
	from the larger cosheaf $\call$ taking the form
	\begin{equation}
		\call/\calk_{d\rhd c} (x + \im\phi_d) = \call_{d\rhd c} (x) + \call_{d \rhd c}(\im \phi_d) 
		= \call_{d\rhd c} (x) + \im\phi_c
	\end{equation}
	where the second equality comes from the commutativity of~\eqref{eq:comm_square_cosheaf}. 
	One can think of $\call / \calk$ as encoding the quotient to $\calk$ inside of $\call$, 
	with a cosheaf projection map $\pi: \call \to \call / \calk$. 
 
 	The relationship between the cosheaves $\calk, \call$ and $\call / \calk$ is especially significant.
	
	\begin{definition}[Exact Sequences]\label{def:exact}
		A sequence of vector spaces
		\begin{equation}
  		\label{eq:exact_vectors}
			\dots \to V_3 \to V_2 \to V_1 \to V_0 \to V_{-1} \to \dots
		\end{equation}
		is said to be {\it exact} if, at each term, the image of the incoming linear transformation is 
		precisely the kernel of the outgoing linear transformation. Equivalently, the 
		sequence~\eqref{eq:exact_vectors} is a chain complex whose homology completely 
		vanishes.
	\end{definition}

	In the special case that \eqref{eq:exact_vectors} is non-zero in at most three incident terms, 
 	one rewrites this as a {\it short exact sequence}
	\begin{equation}
		0 \to V_2 \to V_1 \to V_0 \to 0 .
	\end{equation}
	It follows \cite{AlgebraicHatcher2002} that there are isomorphisms $V_1 \iso V_0 \oplus V_2$ and $V_0 \iso V_1 / 
	V_2$. In the case of cosheaves $\calk, \call$ and $\call / \calk$ as 
	above, one has the short exact sequence of cosheaves
	\begin{equation}
	\label{eq:ses}
		0 \to \calk \xrightarrow{\phi} \call \xrightarrow{\pi} \call/\calk \to 0 .
	\end{equation}
	Over each cell $c$, $\phi_c$ is injective, $\pi_c$ is surjective, and $\im\phi_c 
	\iso \ker\pi_c$. One has an isomorphism of cosheaves $\call \iso \calk \oplus \call / 
	\calk$, in
	which the data of $\call$ is subdivided into its constituent components parallel and 
	orthogonal to $\calk$.
	
	Consolidating cosheaf maps, a short exact sequence of cosheaves induces a short exact 
	sequence of cosheaf chain complexes
	\begin{equation}\label{eq:ses_chains}
		0 \to C \calk \xrightarrow{\phi} C \call \xrightarrow{\pi} 
		C \call/\calk \to 0.
	\end{equation}
	
	These chain complexes yield the canonical {\it long exact sequence} in cosheaf 
	homology
	\begin{equation}\label{eq:les}
		\cdots \rightarrow H_{i+1} \call/\calk \xrightarrow{\vartheta} H_i 
		\calk \xrightarrow{\phi} H_i \call \xrightarrow{\pi} H_i \call/\calk 
		\xrightarrow{\vartheta} H_{i-1} \calk \rightarrow \cdots
	\end{equation}
	where $\vartheta$ here are called {\it connecting homomorphisms}. This sequence is 
	exact as a sequence of vector spaces.
	
	The various maps in the long exact sequence of homology~\eqref{eq:les} can be 
	constructed rigorously. The induced linear maps $\phi$ and $\pi$ are the same as 
	they are in the chain complex setting~\eqref{eq:ses_chains}, only evaluated on 
	homology classes.
	
	For the reader who may not be familiar with methods from homological algebra, we 
	briefly outline the construction of the connecting homomorphisms $\vartheta$ in 
	sequence~\eqref{eq:les} here \cite{AlgebraicHatcher2002}. A homology class in $H_i 
	\call / \calk$ is represented by a cycle in $C_i \call / \calk$. It can be shown that 
	$\vartheta$ 
	maps this representative to a chain in $C_{i-1}\calk$ by the following procedure, 
	outlined 
	in diagram~\eqref{eq:connecting_hom}. The preimage of a cycle in $C_i \call / \calk$ by 
	$\pi$ is a chain in $C_i \call$, to which the boundary map of $\call$ is applied. The 
	preimage of this chain in $C_{i-1} \call$ is an element of $C_{i-1} \calk$, which can be 
	shown to be a cycle, hence a homology class in $H_{i-1} \calk$.
	
	\begin{equation}\label{eq:connecting_hom}
		\begin{tikzcd}
			0 \ar[r] & C_i \calk \ar[r, "\phi"] \ar[d] & C_i \call \ar[r, "\pi"] \ar[d] \arrow[d, 
			dotted, 
			bend left]& C_i \call / \calk 
			\ar[r] \ar[d] \arrow[l, dotted, bend left] & 0\\
			0 \ar[r] & C_{i-1} \calk \ar[r, "\phi"] & C_{i-1} \call \ar[r, "\pi"] \arrow[l, dotted, 
			bend 
			left] & C_{i-1} \call / \calk 
			\ar[r] & 0
		\end{tikzcd}
	\end{equation}
	
	Furthermore, incorporating $\vartheta$ it can be shown the sequence~\eqref{eq:les} 
	truly is exact.
	
	\begin{example}[Geometry of Axial Forces]
 	\label{ex:force_embed}
		The edge-stalks of the force cosheaf $\calf$ can, by abuse of notation, be thought of 
		as	edge-parallel subspaces of the vertex stalks $\R^n$. Consequently, $\calf$ can be 
		embedded as a sub-cosheaf of the larger constant cosheaf $\overline{\R^n}$. This 
		latter constant cosheaf describes the ambient Euclidean space of forces.
		
		Suppose $\calf$ is the force cosheaf over a two dimensional cell complex $X$ realized 
		in $\R^n$. There is an injective map $\phi: \calf \to \overline{\R^n}$ between 
		cosheaves where $\phi_v$ is the identity over vertices and $\phi_e$ is the 
		embedding map $E_e: \R \iso {\bf e} \to \R^n$ of the edge stalk into $\R^n$. This 
		cosheaf map $\phi_f=0$ is trivial over faces. 
		
		From the injective cosheaf map $\phi$, the quotient $\overline{\R^n} / \calf$ is the 
		cosheaf of complementary data in $\R^n$. The stalks of this cosheaf are comprised of 
		full $\R^n$ data over faces and orthogonal data ${\bf e}^\perp$ along edges, with 
		trivial data over vertices. Summarizing the relationships between these is a commutative 
		diagram of stalks over every triplet of incident cells $v\lhd e \lhd f$.
  	\begin{equation}
   	\label{eq:2D force diag}
			\begin{tikzcd}
				& 0 \ar[r]  \ar[d] & \overline{\R^n}_f \ar[r, "\id"]\ar[d, "\id"] & 
				(\overline{\R^n}/\calf)_f \ar[r] \ar[d, "P_{e^\perp}"] & 0 \\
				0 \ar[r] & \calf_e \ar[r, "E_e"] \ar[d, "E_e"] & \overline{\R^n}_e \ar[r, 
				"P_{e^\perp}"] \ar[d, "\id"] & (\overline{\R^n}/\calf)_e \ar[r] \ar[d] & 0\\
				0 \ar[r] & \calf_v \ar[r, "\id"] & \overline{\R^n}_v \ar[r] & 0 &
			\end{tikzcd}
	\end{equation}
		Here, $P_{e^\perp}$ is the projection map to the equivalence class $x \to x + {\bf e} 
		\in \R^n/{\bf e}$ in the quotient. This can be regarded as projecting to the orthogonal 
		subspace ${\bf e^\perp}$, utilizing an inner product on stalks.
		
		Choosing $X$ to be a polyhedral sphere $S^2$, the Poincar\'e dual of the cosheaf 
		$\R^n / \calf$ is a position sheaf over the dual polyhedral sphere $\tilde{X}$.
  		Interpreting face stalks of $\overline{\R^n} / \calf$ to be the coordinates of 
  		Poincar\'e dual vertices in Euclidean space, this forms a position sheaf over 
  		$\tilde{X}$.
	\end{example}
	
	\section{Planar Graphic Statics}
	\label{sec:gs2}
	At heart graphic statics is {\it visual}; 
	hidden internal axial forces within truss members are put on full display in reciprocal
	diagrams. This phenomenon is most clearly seen with planar trusses, where a spherical cell 
	complex is projected to the plane.
 
	
	\begin{figure}[!ht]\centering
		\begin{subfigure}[t]{0.32\textwidth}\centering
			\includegraphics{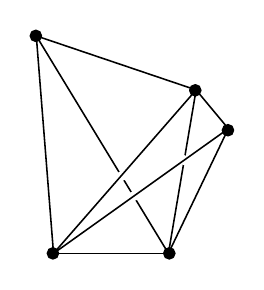}
		\end{subfigure}
		\begin{subfigure}[t]{0.32\textwidth}\centering
			\includegraphics{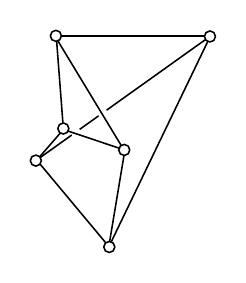}
		\end{subfigure}
		\begin{subfigure}[t]{0.32\textwidth}\centering
			\includegraphics{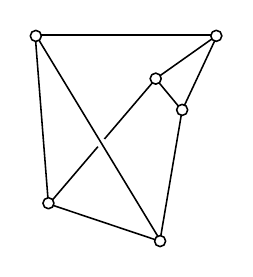}
		\end{subfigure}
		\caption{Force diagrams dual to a form diagram are pictured. The form diagram (left) 
			has two degrees of self stress generated by holding the covered edges in tension or 
			compression. One type of dual diagram (center) emerges from specifying these forces to 
			have the same sign. Another type of dual diagram (right) comes from having one edge 
			in tension and one in compression. As force configurations can be linearly added over 
			a truss, force diagrams can likewise be added, giving a ``structural algebra'' of 
			graphs.}
		\label{fig:planar}
	\end{figure}
	
	\subsection{2D Planar Duality}
	\label{sec:2D1}
	
	Here we formulate graphic statics algebraically. From Example~\ref{ex:force_embed}, 
	over a form diagram $(X,p)$ in $\R^2$ there is a short exact sequence of cosheaves
	\begin{equation}\label{eq:2D force SES}
		0 \to \calf \xrightarrow{\phi} \overline{\R^2} \xrightarrow{\pi} \overline{\R^2}/\calf 
		\to 
		0
	\end{equation}
	where $\calf$ is the force cosheaf and $\overline{\R^2}$ the constant cosheaf with 
	$\R^2$ stalks. The homological algebra of these cosheaves {\it is} graphic statics in its 
	abstract form. With the cell complex $X$ a topological sphere $S^2$, the Poincar\'e 
	dual sheaf to $\overline{\R^2} / \calf$ takes the form of a position sheaf over 
	$\tilde{X}$ whose structure is determined by the form diagram $(X,p)$. For
  brevity,
  we use a more concise notation $\calg = \overline{\R^2}/\calf$ .
	
	With the interpretations of Section~\ref{sec:position}, an element of the stalk 
	$\tilde{\calg}_{\tilde{f}}$ gives coordinates for the vertex $\tilde{f}$ dual to
  the face $f$. A cochain of $C^0 \tilde{\calg}$ is then the data of a
  realization $q$ of $\tilde{X}$ (by abuse of 
	notation, interchanging deformations and realizations of diagrams, discussed in 
	Figure~\ref{fig:position_square}). To be a cocycle, $q$ must satisfy a cocycle condition 
	property over dual edges. For coordinates $q_{\tilde{f}}$ and $q_{\tilde{g}}$
  of dual vertices incident to an edge $\tilde{f}, \tilde{g} \lhd \tilde{e}$, 
	the difference between these coordinates must be in the subspace ${\bf e} \subset 
	\R^2$. Under these constraints, a realization $q$ would set a dual edge $\tilde{e}$ {\it 
	parallel} to its primal counterpart $e$ in the diagram $(X,p)$. Cohomology classes of 
	$H^0 \tilde{\calg}$ are dual realizations specifying {\it parallel dual diagrams} 
	$(\tilde{X},q)$.
	
	The short exact sequence of cosheaves~\eqref{eq:2D force SES} above has a 
	corresponding long exact homology sequence~\eqref{eq:les}. Due to the trivial stalks 
	of $\calf$ and $\calg$, this 
	long exact sequence is simplified by two zero homology spaces $H_2 \calf$ and $H_0 
	\calg$. Further simplifications are made by considering the homology of 
	the constant cosheaf. These homology spaces are determined by the base topology of 
	$S^2$, namely $\dim H\overline{\R^2} = 2 \dim H S^2$. In particular $H_1 
	\overline{\R^2} = 0$ is zero, splitting the long exact sequence into two.
	\begin{equation}\label{eq:2D1}
		0 \to \R^2 \xrightarrow{\pi} H_2 \calg \xrightarrow{\vartheta} H_1\calf \to 0
	\end{equation}
	\begin{equation}\label{eq:2D2}
		0 \to H_1 \calg \xrightarrow{\vartheta} H_0 \calf \xrightarrow{\phi} \R^2 \to 0
	\end{equation}
	where $\vartheta$ are connecting homomorphisms. These two homology sequences 
	are at the heart of two-dimensional graphic statics. The first encodes the algebraic
	relationship between axial self stresses and dual realizations. The second links 
	mechanical degrees of freedom to dual impossible rotations.
	
	\begin{theorem}[Planar Force and Form \cite{ReciprocalMaxwell1864, 
			Crapo1993}]\label{thm:plane2D}
		Let $(X,p)$ be a spherical form diagram in $\R^2$. There is a bijection between self 
		stresses $a \in	H_1 \calf$ and parallel force diagrams $(\tilde{X},q)$ with realization 
		$q\in H^0 \tilde{\calg}$ up to global translation satisfying the condition
    that over each edge of $X$ bounded by 
		two faces $e\lhd f,g$, the force $w_e$ is equal to the oriented length $[e:
		f](q_{\tilde{f}}-q_{\tilde{g}})$ of the dual edge $\tilde{e}$.
	\end{theorem}
	\begin{proof}
		This theorem is a consequence of the exactness of sequence~\eqref{eq:2D1}. We 
		walk through each step of the process to fully develop the correspondence.
		
		By exactness at $H_2 \overline{\R^2}\iso \R^2$ the map $\pi$ is injective, and by 
		exactness at $H_1 \calf$ the map $\vartheta$ is surjective. In particular,
    for any self stress 
		$a\in H_1 \calf$ over $X$ there must be a realization $q\in H^0 \tilde{\calg}$ as a 
		preimage of $\vartheta$. For any two preimages $q, r$ of $w$ the realization	$q-r$ 
		is in the kernel of the connecting homomorphism $\vartheta$. Exactness at $H_2 
		\calg$ states that the kernel of $\vartheta$ is equal to the image of $\pi$, so $q-r$ 
		itself has a preimage by 
		$\pi$ in $H_2 \overline{\R^2} \iso \R^2$ a global translation 
		vector. Subsequently, the difference of realizations $q-r$ is a global translation of all 
		dual vertex coordinates of $\tilde{X}$. This describes the bijection up	to global 
		translation.
		
		The value of the bijection follows from the construction of the connecting 
		homomorphism $\vartheta$, outlined by line~\eqref{eq:connecting_hom}. The 
		preimage 
		of $q$ by $\pi$ is a cycle in $C_2 \overline{\R^2}$, where it is mapped by the
		boundary map of the cosheaf $\overline{\R^2}$ into $C_1 \overline{\R^2}$. There it 
		takes value $[e:f](q_f - q_g)$ over edge $f,g\rhd e$, which, over the dual diagram is 
		the difference in coordinates $[\tilde{f}: \tilde{e}] (q_{\tilde{f}} - q_{\tilde{g}})$ over 
		the dual edge $\tilde{e}$. This is the value of the self-stress $w_e$ over the primal 
		edge $e$.
		
		In this way, the dual edge $\tilde{e}$ is not only parallel to $e$ but its length is 
		determined by the cycle $w\in H_1 \calf$. The direction and magnitude of the force 
		vector $w_e$ is equal to the edge $\tilde{e}$ as realized by $q\in H^0 \tilde{\calg}$.
	\end{proof}
	Theorem~\ref{thm:plane2D} describes the nature of the isomorphism $H^0 
	\tilde{\calg} / \R^2 \iso H_1 \calf$ following the exact sequence of vector 
	spaces~\eqref{eq:2D1}. The second exact sequence~\eqref{eq:2D2} gives a similar 
	theorem linking two other homological spaces. Recall that an impossible edge rotation 
	is an edge assignment that cannot possibly result from any realization even allowing for 
	free axial extension of dual edges. Here, impossible dual edge rotations of the 
	reciprocal diagram $(\tilde{X},q)$ are linked to mechanical freedom of the primal.
	
	\begin{theorem}[Mechanics and Impossible Rotations in $\R^2$
		]\label{thm:second_theorem}
		Let $(X, p)$ be a spherical diagram realized in $\R^2$ and let 
		$(\tilde{X}, q)$ be a parallel reciprocal diagram. There is a bijection 
		correspondence between impossible edge rotations of $(\tilde{X},q)$ and the 
		mechanisms and global rotations of $(X,p)$.
	\end{theorem}
	\begin{proof}
		Every impossible edge rotation of $\tilde{X}$ is an element of $H^1 \tilde{\calg} \iso 
		H_1 \calg$. The map $\vartheta$ in exact sequence~\eqref{eq:2D2} is injective, 
		meaning 
		taking the influence of dual edge rotations and summing them at vertices activates a 
		degree of freedom of $X$ in 
		$H_0 \calf$. The map $\phi: H_0 \calf \to \R^2$ is surjective, meaning global 
		translations are always degrees of freedom of $X$.
		
		By exactness at the middle vector space $H_0 \calf$, the image of $\vartheta$ in 
		$H_0 \calf$ is isomorphic to the kernel of $\phi$, which consists of degrees of 
		freedom of $X$ which have no net global translation effect. Because any linear map is 
		an isomorphism onto its image, $H^1 \tilde{\calg}$ is isomorphic to this space of 
		non-translational degrees of freedom, equivalent to the space of mechanisms and 
		global rotations of $X$.
	\end{proof}
	
	A mechanism or global rotation $x\in H_0 \calf$ applied to the structure $X$ rotates at 
	least some of the edges of $(X,p)$. These edge rotations of may be 
	transferred to the reciprocal $(\tilde{X},q)$ after accounting for the {\it sign} of 
	dual edges. By Theorem~\ref{thm:plane2D}, each dual edge $\tilde{e}$ of 
	$(\tilde{X},q)$ corresponds to a primal edge $e$ of $(X,p)$ in tension or in compression. 
	It is a technical point, but if an edge $e$ is in tension (under the convention that tension 
	is a negative value over edges), its dual $\tilde{e}$ is realized in reverse, inverted in the 
	force diagram. Consequently, cycles or cocycles drawn over mirrored dual edges 
	$\tilde{e}$ are also drawn mirrored. A clockwise rotation over an edge in tension 
	transfers to a 	counterclockwise rotation over its counterpart.
	
	\begin{example}[Reciprocal Rotations]\label{ex:rotation transfer}
		Suppose that $(X,p)$ is the form diagram in Figure~\ref{fig:force_boxed}. 
		This truss has two translational degrees of freedom and one rotational degree of 
		freedom. Under a counterclockwise rotation around the central node pictured in 
		Figure~\ref{fig:rotation_scale} (left), all edges rotate counterclockwise as well. This 
		form diagram has one degree of self-stress holding the internal four 
		edges in tension. By Theorem~\ref{thm:plane2D}, this self stress determines a 
		parallel realization of the dual cell complex $\tilde{X}$, pictured (center) in 
		Figure~\ref{fig:rotation_scale}.
		
		Because the internal four diagonal edges (left) are loaded in tension, their dual edges 
		(center) have their sign reversed. The rotations of these four diagonal edges on the 
		dual diagram (center) are reversed, rotating clockwise instead. This is an impossible 
		edge rotation of the force diagram -- an element of $H^1 \tilde{\calg}$ over 
		$\tilde{X}$. This matches with previously found impossible rotation configuration of 
		Figure~\ref{fig:position_boxed}.
	\end{example}
	
	\subsection{Plane Reciprocity}
	\label{sec: reciprocity}
	The concept of ``reciprocity'' here refers to interchanging the roles of the
	form diagram and force diagram, reversing the relationship between the two.
	The self stresses of any form diagram determine dual parallel realizations according to 
	Theorem~\ref{thm:plane2D}. A dual diagram $(\tilde{X},q)$ itself is a diagram in 
	$\R^2$, with all the characteristics that the prior $X$ had. Namely $(\tilde{X},q)$ may 
	have self stresses and mechanisms, governed by the force cosheaf over it.
	
	We may proceed with constructing the force cosheaf over $\tilde{X}$, but we would 
	like to do so with respect to the cosheaves $\calf, \overline{\R^2}, \calg$ previously 
	developed in Section~\ref{sec:2D1}. In this way, the homological properties of 
	the force diagram can be interconnected to those of the form diagram. Here, we 
	introduce the technique of rotating dual realizations. For $q$ an parallel dual realization 
	in $H^0 \tilde{G}$, let $q'$ denote the same realization of $\tilde{X}$ {\it rotated one 
	quarter turn} in the plane. By employing $q'$, dual edges would be realized 
	perpendicular 
	to their counterparts in the form diagram $(X,p)$. The benefit is that the cosheaf 
	$\tilde{\calg}^\vee$ {\it is} the force cosheaf over $(\tilde{X},q')$. Respecting the duality 
	between reciprocal force and form diagrams, we call $\tilde{\calg}^\vee$ the {\it 
	reciprocal cosheaf} to $\calg$.
	
	\begin{lemma}\label{lem:force_position_iso}
		Let $(X,p)$ be a diagram in $\R^2$. The force cosheaf $\calf$ and the 
		linear dual of the position sheaf $\calj^\vee$ are isomorphic.
	\end{lemma}
	\begin{proof}
		We construct a cosheaf isomorphism $\mu: \calf \to \calj^\vee$. Over edges $\mu$ 
		is set to the identity transformation (taking vectors to linear functionals) and is a 
		quarter turn over vertex stalks. Over incident edges and vertices $v\lhd e$, the 
		following diagram commutes
		\begin{equation}
			\begin{tikzcd}
				\calf_e \ar[r, "\id"] \ar[d, "E_e"] & \calj^\vee_e \ar[d, "E_{e^\perp}"] \\ 
				\calf_v \ar[r, "\mu_v"] & \calj^\vee_v
			\end{tikzcd}
		\end{equation}
		where $\mu_v$ maps the subspace parallel to the edge ${\bf e}\subset \R^2$ to the 
		perpendicular subspace ${\bf e^\perp} \subset \R^2$ by one-quarter turn (either 
		clockwise or counter-clockwise, but consistent across all vertices).
	\end{proof}
	
	The technique of rotating vector values has been previously utilized to show the 
	equivalence of rotations and (scaling) of a structure about the center of rotation 
	\cite{MechanismsMitchell2016}. This is evident in Figure~\ref{fig:rotation_scale} (left) 
	and 
	(right).
	
	\begin{figure}[th]
		\begin{subfigure}[t]{0.3\textwidth}\centering
			\includegraphics{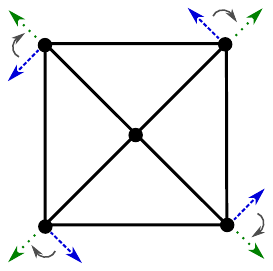}
		\end{subfigure}
		\begin{subfigure}[t]{0.38\textwidth}\centering
			\includegraphics{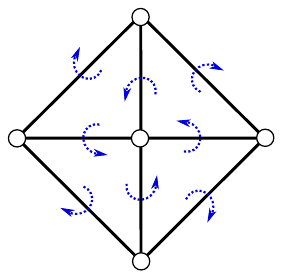}
		\end{subfigure}
		\begin{subfigure}[t]{0.3\textwidth}\centering
			\includegraphics{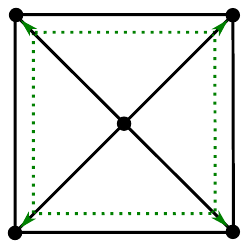}
		\end{subfigure}
		\caption{The truss $(X,p)$ from Figure~\ref{fig:force_boxed} is rotated 
			counterclockwise (left). This system of edge rotations on $(X,p)$ can be transferred 
			to the dual diagram $(\tilde{X}, q)$ (center). Rotating the chain (left) results in a 
			parallel deformation (right).}
		\label{fig:rotation_scale}
	\end{figure}
	
	Both parallel and perpendicular drawing conventions $q$ and $q'$ of $\tilde{X}$ 
	have been widely utilized \cite{Williams2016}. Where dual edges are 
	drawn parallel is known as the {\it Cremona convention} while where dual 
	edges are drawn perpendicular is known as {\it Maxwell's convention}.
	
	Note the sheaf $\tilde{\calg}$ has been predetermined to be the position sheaf over 
	$\tilde{X}$ with respect to a realization $q$ parallel to $p$. By 
	Lemma~\ref{lem:force_position_iso}, its reciprocal $\tilde{\calg}^\vee$ is the force 
	cosheaf over $\tilde{X}$ with respect to the realization $q'$ after rotating dual vertex 
	coordinates. The reciprocal cosheaf exact sequence to sequence~\eqref{eq:2D force SES} 
	is
	\begin{equation}\label{eq:2D dual SES}
		0\to \tilde{\calg}^\vee \xrightarrow{\pi^\vee} \overline{\R^2}
		\xrightarrow{\phi^\vee} \tilde{\calf}^\vee \to 0
	\end{equation}
	where $\overline{\R^2}$ is the constant cosheaf over $\tilde{X}$. Note 
	that the adjoint maps not only flip the domain and codomains of linear incidence maps 
	internal to (co)sheaves, but also flip the direction of linear (co)sheaf maps in the exact 
	sequence.
	
	With $\tilde{\calg}^\vee$ the force cosheaf over $(\tilde{X}, q')$ and $\overline{\R^2}$ 
	the two dimensional constant cosheaf, sequence~\eqref{eq:2D dual SES} is 
	fundamentally the same as the original cosheaf sequence~\eqref{eq:2D force SES}, only 
	over $\tilde{X}$ instead of $X$. Consequently, the double Poincar\'e dual sheaf 
	$\doubletilde{\calf}^\vee = \calf^\vee$ {\it is} a position sheaf. The cohomology 
	$H^0 \calf^\vee$ of consists of alternative realizations $p'$ of $X$ parallel to $q'$. 
	These 
	new diagrams $(X,p')$ are thus perpendicular to the starting diagram $(X,p)$. This can 
	also be seen from employing Lemma~\eqref{lem:force_position_iso} directly from the 
	original force cosheaf $\calf$ over $(X,p)$.
	
	The consequence of reciprocal and symmetric cosheaf sequences is this. 
	Theorems~\ref{thm:plane2D} and~\ref{thm:second_theorem} both apply in the 
	reciprocal setting, from diagrams $(\tilde{X},q')$ to $(X,p')$. After rotating by one 
	quarter 
	turn, their relationship is the same as that of $(X,p)$ and $(\tilde{X},q)$. This symmetry 
	allows the combination of the theorems combining their effects.

	\begin{theorem}[Reciprocal Relations]\label{thm:plane reciprocity}
		Suppose $(X, p)$ and $(\tilde{X}, q)$ are planar reciprocal diagrams over 
		the plane $\R^2$, both with spherical $S^2$ regular cellular topology. There are 
		isomorphisms between the following vector spaces:
		\begin{itemize}
			\item[(i)] The space of mechanisms and global rotations of $(X,p)$.
			\item[(ii)] The space of parallel deformations of $(X,p)$ up to global translation.
			\item[(iii)] The space of impossible edge rotations over $(\tilde{X}, 
			q)$.
			\item[(iv)] The space of pure self shear stresses over $(\tilde{X}, q)$.
			\item[(v)] The space of axial self stresses over $(\tilde{X}, q)$.
		\end{itemize}
	\end{theorem}
	\begin{proof}
		We will prove the equivalence of (i) and (ii) on the diagram $(X,p)$, then prove the 
		equivalences on the diagram $(\tilde{X}, q)$.
		
		By Lemma~\ref{lem:force_position_iso} and Theorem~\ref{thm:linear_homology}, 
		the 
		homology of the force cosheaf $\calf$ over $(X,p)$ is 
		isomorphic to the homology of the position sheaf $\mu \calf^\vee$ over diagram 
		$(X,p)$. Here, movement from a mechanism or rotation system (i) in $H_0 \calf$ is 
		rotated by one quarter turn to form an parallel axial extension (ii) in $H^0 
		\mu \calf^\vee$. This latter homology space is equivalent to the space of parallel 
		deformations of $p$ (ii). Removing the global translational degrees of freedom results 
		in the expected isomorphism $H_0 \calf - \R^2 \iso H^0 \mu \calf^\vee / \R^2$.
		
		Recall that $\tilde{\calg}$ is the position sheaf over $(\tilde{X}, q)$. The space of 
		impossible edge rotations $H^1 \tilde{\calg}$ is isomorphic to the space of pure shear 
		stresses $H_1 \tilde{\calg}^\vee$ as described Example~\ref{ex:self shear}. 
		Lemma~\ref{lem:force_position_iso} shows that this cosheaf $\tilde{\calg}^\vee$ is 
		isomorphic to the force cosheaf $\mu^{-1}\tilde{\calg}^\vee$ over $(\tilde{X}, q)$ 
		after rotating by one quarter turn. This gives an isomorphism between the space (v) 
		of axial self stresses $H_1 \mu^{-1} \tilde{\calg}^\vee$ and (iv).
		
		Lastly, we connect the quantities over the two diagrams. 
		Theorem~\ref{thm:second_theorem} draws the 
		equivalence between degrees of freedom $H_0 \calf$ excluding translations and the 
		homology space $H^1 \tilde{\calg}$ characterizing impossible dual rotations on the 
		dual diagram $(\tilde{X},q)$. This is an isomorphism between (i) and (iii). Similarly, 
		Theorem~\ref{thm:plane2D} gives the equivalence between axial self-stresses and 
		dual parallel realizations (equivalent to parallel deformations, as discussed in 
		Section~\ref{sec:position} and Figure~\ref{fig:position_square}) up to translation. This 
		is an isomorphism between (ii) and (v). The following is a commutative diagram of the 
		isomorphisms at play:
		\begin{equation}
			\begin{tikzcd}
				H_0 \calf - \R^2 \ar[r, "\vee"] & H^0 \calf^\vee/\R^2 \ar[r, "\mu"] \ar[l] \ar[d, 
				"\tilde{\vartheta}^\vee"] & H^0 \mu\calf^\vee/\R^2 \ar[d, 
				"-\tilde{\vartheta}"] \\
				H^1 \tilde{\calg} \ar[u, "\tilde{\vartheta}"] \ar[r, "\vee"] & H_1\tilde{\calg}^\vee 
				\ar[l] & H_1 \mu^{-1}\tilde{\calg}^\vee \ar[l, "\mu"]
			\end{tikzcd}
		\end{equation}
	\end{proof}
	
	Theorem~\ref{thm:plane reciprocity} is indifferent to whether $(X, p)$ is considered as 
	a form diagram or $(\tilde{X}, q)$ is. These relationships are best understood visually, 
	with example in Figure~\ref{fig:reciprocal_mechanism}.
	
	\begin{figure}[!ht]
		\begin{subfigure}[t]{0.24\textwidth}\centering
			\includegraphics{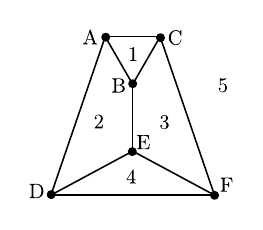}
			\caption{}
		\end{subfigure}
		\begin{subfigure}[t]{0.24\textwidth}\centering
			\includegraphics{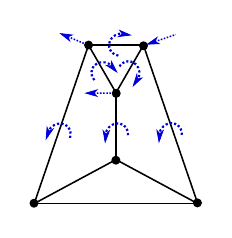}
			\caption{}
		\end{subfigure}
		\begin{subfigure}[t]{0.24\textwidth}\centering
			\includegraphics{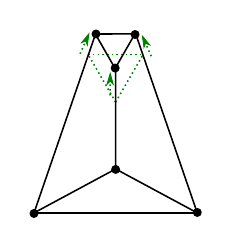}
			\caption{}
		\end{subfigure}
		\begin{subfigure}[t]{0.19\textwidth}\centering
			\includegraphics{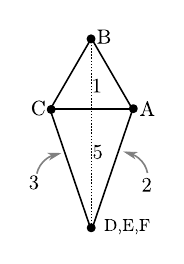}
			\caption{}
		\end{subfigure}
		\begin{subfigure}[t]{0.29\textwidth}\centering
			\includegraphics{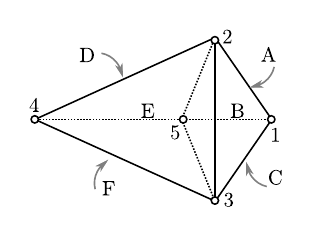}
			\caption{}
		\end{subfigure}
		\begin{subfigure}[t]{0.39\textwidth}\centering
			\includegraphics{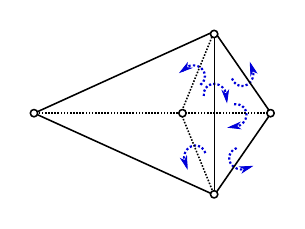}
			\caption{}
		\end{subfigure}
		\begin{subfigure}[t]{0.29\textwidth}\centering
			\includegraphics{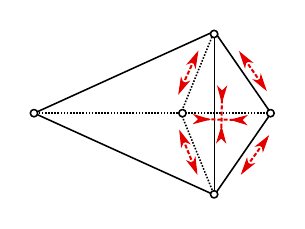}
			\caption{}
		\end{subfigure}
		\caption{From a form diagram $(X,p)$ (a), a force diagram $(\tilde{X}, 
		q)$	(e) is drawn by holding the five interior edges of (a) in tension and the outside 
		four in compression. We investigate the equivalent spaces to the mechanism (b) as 
			described by Theorem~\ref{thm:plane reciprocity}. The motion of vertices, shifted 
			by one quarter turn, are a parallel edge extension (c). The length changes of 
			each member form the parallel realization (d). Over the dual, the mechanism (b) 
			translates to a set of impossible edge rotations (f) by 
			Theorem~\ref{thm:second_theorem}. These rotations in (f) are equivalent to a 
			degree of self shear over the force diagram. Shifting these dual edge rotations by 
			one 
			quarter turn in the plane forms an axial self stress (g). By reciprocity and 
			Theorem~\ref{thm:plane2D}, (g) generates the parallel diagram (d).}
		\label{fig:reciprocal_mechanism}
	\end{figure}

	\subsection{The Polyhedral Lifting Correspondence}
	\label{sec:polyhedral_liftings}
	
	Dating back to Maxwell's original papers \cite{ReciprocalMaxwell1864, maxwell_1870}, 
	self stresses over a loaded truss not only correspond to a dual force diagram, but also a 
	``lifting'' of cell structure to a polyhedron lying ``above'' the 
	original truss in three dimensions. Because the base cell complex in planar graphic 
	statics is spherical, the faces between the members of a loaded truss lift to a 
	topologically spherical polyhedron \cite{ParadigmHopcroft1992}. See 
	Figure~\ref{fig:affine_boxed} for an example of a lift over the self stressed cell complex 
	of Figure~\ref{fig:force_boxed}. This lifting relation is commonly denoted the {\it 
	Maxwell--Cremona correspondence} \cite{erickson2020toroidal} and the lifted 
	polyhedron itself has been shown to be a {\it discrete Airy stress function} over the 
	underlying graph \cite{MechanismsMitchell2016}.
	
	\begin{figure}[ht]\centering
		\includegraphics{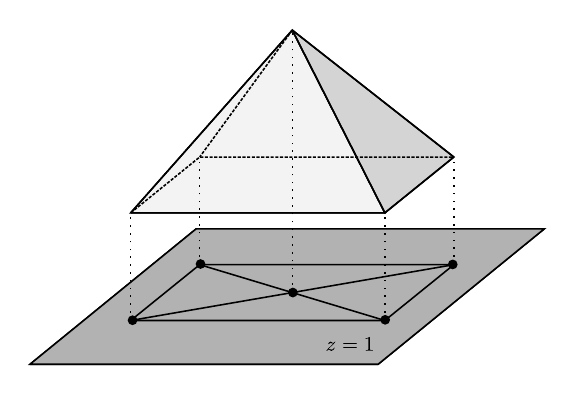}
		\caption{A lifting of the the cell complex in Figure~\ref{fig:force_boxed}
			derived from the dimension of self stress.}
		\label{fig:affine_boxed}
	\end{figure}
	
	A face lifted from $\R^2$ to $\R^3$ is modeled as the graph of an {\it affine 
	function} over the face. Recall that an affine function on $\R^n$ with basis $\{x_i\}$ 
	canonically takes the form $h(x_1, \dots, x_n) = \lambda_1x_1 + \dots + \lambda_n x_n 
	+ 
	\lambda_{n+1}$ for real parameters 
	$\lambda_i$. The collection of these
	functions define the $n+1$ dimensional vector space $A\R^n$ of affine
	functions over $\R^n$. An equivalent perspective is that a affine function is a linear 
	functional in $(\R^{n+1})^\vee$ evaluated on the hyperplane $H = \{x_{n+1}=1\}\subset
	\R^{n+1}$.
	\begin{lemma}{\cite{ParadigmHopcroft1992}}\label{lem:affine_functionals} 
		There is an isomorphism $\eta: (\R^{n+1})^\vee \to A\R^n$ given
		by $\eta(h) = h|_H$, restriction of the function $h$ to the hyperplane $H$.
	\end{lemma}
	
	Here we let $X$ be realized in the hyperplane $H\subset \R^3$. Lifting $X$ from $H$ to 
	a polyhedron in $\R^3$ requires lifting each face such that these {\it glue together} in 
	the following sense. If $f$ and $g$ are neighboring faces incident to an edge $e$, then 
	the graphs of two affine functions $h_f$ and $h_g$ over these faces must meet and be 
	equal over the edge $e$. The difference $h_f - h_g$ will be the zero function restricted 
	to the edge $e$, with gradient in $H$ perpendicular to $e$.
	
	A collection of affine functions over faces $h = \{h_{f_i}\}$ is a lift of a diagram $(X,p)$ if 
	and only if pairs of affine functions glue together over all conjoining edges. This is a 
	cycle over the {\it affine cosheaf} $\cala$ yet to be defined. The homology $H_2 \cala$ 
	will consists of all polyhedral lifts of the underlying cell complex $X$.
	
	In a polyhedral lift, not only faces but edges and vertices are lifted as well to $\R^3$. 
	Modeling a lift of all cells, the cosheaf $\cala$ has stalks comprised of spaces of 
	locally defined affine functions, namely an isomorphism $\cala_c \iso 
	A\R^{\textrm{dim }c}$. To formalize this, let stalks $\cala_c$ consist of {\it equivalence 
	classes of affine functions} where $h_c \sim k_c$ if $h_c - k_d$ is the zero function 
	when 
	restricted to the cell $c$. Functions in the same equivalence class differ only away from 
	their respective cell $c$. Under cosheaf extension maps $\cala_{d \rhd c}$, equivalence 
	classes merge as constraints from the underlying cell geometry becomes less 
	restrictive. With stalks consisting of quotient vector spaces, the structure of $\cala$ is of 
	a quotient cosheaf.
	
	\begin{definition}[Zero Locus Cosheaf]
		Suppose $(X,p)$ is a diagram in the hyperplane
		$H\subset \R^{n+1}$. Define $\calz$ as the {\it zero locus cosheaf} with
		stalks
		\begin{equation}
			\calz_c = \textrm{span} \{ h\in A\R^n : h(p_v) = 0 \textrm{ for every vertex } v\lhd c 
			\}
		\end{equation}
		
		of affine functions that are zero on the realization of $c$. Extension maps
		$\calz_{c \lhd d}$ are inclusions of subspaces $\calz_d\to \calz_c$.
	\end{definition}
	
	The zero locus cosheaf $\calz$ has stalks of functions that vanish over their respective 
	cells. Through the lens of Lemma~\ref{lem:affine_functionals}, an affine function in 
	$\calz_c$ is equivalent to a linear functional with the cell $c$ (in $H$) contained in its 
	null space. It is best to require the vertex coordinates to have full dimensional span 
	$\dim c + 1$, otherwise the cell is degenerate. In dimension 2, specifying no degeneracy 
	means no lifted faces are vertical with infinite gradient.
	
	\begin{definition}[Affine Cosheaf]
		There is an injective map $\phi: \calz \to \overline{A\R^n}$ taking each
		subspace of affine functions $\calz_c$ into the total space of affine
		functions $\overline{A\R^n}_c$. Define the {\it affine cosheaf} $\cala$ as the
		quotient of this inclusion map.
	\end{definition}
	
	\begin{figure}[ht]\centering
		\includegraphics{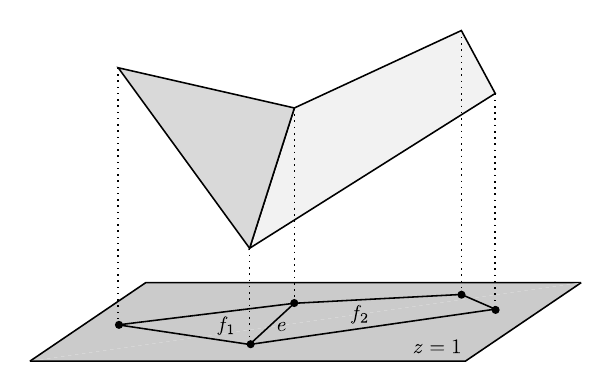}
		\caption{A local view of the affine cosheaf $\cala$. There are two affine
			functions in $A\R^2$ restricted to the faces $f_1$ and $f_2$. The graphs
			of these functions are pictured. These functions coincide where they
			intersect over $e$.}
		\label{fig:affine_function}
	\end{figure}
	
	The affine cosheaf $\cala$ in dimension $n=2$ is most relevant to graphic statics. 
	With the association $\cala = \overline{A\R^2}/\calz$, elements of the stalks of $\cala$ 
	are equivalence classes, with each class consisting of functions that differ by an element 
	of $\calz$. Because stalks $\calz_f$ are zero over non-degenerate faces, 2-chains of 
	$\cala$ are regular affine functions. A cycle $h\in H_2 \cala$ is a 
	collection of these affine functions over faces where the difference of functions 
	$h_f-h_g$ 
	is in the zero class in $\cala_e$, or an element of $\calz_e$, over every set $e\lhd f,g$. 
	The piece-wise graph of the 2-cycle $a$ then is a polyhedron in $\R^3$ 
	with lifted faces glued along their boundary edges. Elements of $H_2 \cala$ are {\it 
		polyhedral liftings} of $(X,p)$.
	
	\begin{lemma}\label{le: f z isomorphism}
		\label{lem:affine iso} The force cosheaf $\calf$ over the hyperplane $H\subset 
		\R^3$ and the cosheaf
		$\calz$ over $A\R^2$ are isomorphic.
	\end{lemma}
	\begin{proof}
		
		Here we define the cosheaf isomorphism $\tau: \calf \to \calz$. Over an edge
		with incident vertices $u,v\lhd e$, the map $\tau_e$ is given by
		\begin{equation}
			\tau_e(w) = w (p_u \times p_v)^\vee
		\end{equation}
		where $\times$ is the cross product on vectors in $\R^3$ and $w \in \calf_e
		\iso \R$ is a scalar representing a multiple of the vector $p_u-p_v$. We further define 
		$\tau_v$ as
		\begin{equation}
			\tau_v(b) = (b \times p_v)^\vee
		\end{equation}
		over vertex $v$ where $b\in \calf_v \iso \R^2 \times \{1\}$ is a vector in
		$H\subset \R^3$. By counting dimensions of the resulting vector space of
		functions, $\tau$ is an isomorphism over each cell. Two calculations
		\begin{equation}
			\tau_u\circ \calf_{e\rhd u}(w)= (w(p_u-p_v) \times
			p_u)^\vee = w (p_u \times p_v)^\vee = \calz_{e \rhd u}\circ \tau_e(w)
		\end{equation}
		\begin{equation}
			\tau_v\circ \calf_{e\rhd v}(w) = (w(p_u-p_v) \times
			p_v)^\vee = w (p_u \times p_v)^\vee = \calz_{e\rhd v}\circ \tau_e(w)
		\end{equation}
		prove that $\tau$ indeed is a cosheaf isomorphism.
	\end{proof}
	
	A corollary to Lemma~\ref{le: f z isomorphism} is the existence of a short exact 
	sequence
	\begin{equation}
		0 \to \calf \xrightarrow{\phi\circ \tau} \overline{A\R^2} \to \cala \to 0.
	\end{equation}
	
	Note that for any edge with endpoints $v,u\lhd e$, for any two height values of $v$ and 
	$u$ there exists an affine function $h_e$ over $e$ which takes these values over $v$ 
	and 
	$u$. Because there is always an edge connected to each vertex, the 
	boundary map $\bdd: C_1 \cala \to C_0 \cala$ is surjective and $H_0 \cala$ is
	trivial. Examining the resulting long exact sequence, we are greeted by a pair of short 
	exact sequences.
	
	\begin{equation}\label{eq:affine_1}
		0\to A\R^2 \to H_2 \cala \to H_1 \calf \to 0
	\end{equation}
	\begin{equation}\label{eq:affine_2}
		0 \to H_1 \cala \to H_0 \calf \to A\R^2 \to 0
	\end{equation}
	
	Sequence~\eqref{eq:affine_1} gives rise to the Maxwell--Cremona lifting correspondence.
	
	\begin{theorem}[The Maxwell--Cremona Correspondence
		\cite{ReciprocalMaxwell1864, Whiteley1982a,
			ParadigmHopcroft1992}] \label{thm:maxwell_lift}
		Let $(X,p)$ be spherical diagram in the projective plane $H \subset \R^3$. There is a 
		bijection between self stresses $w\in H_1 \calf$ and polyhedral liftings $h\in H_2 
		\cala$ up to shifts of a global affine function. Over an edge bounded by two faces
		$e\lhd f,g$, the force $w_e$ over each edge is equal to
		\begin{equation}\label{eq:grad_force}
			w_e = [e:f] \frac{h_f(p_*) - h_g(p_*)}{\det[p_u,p_v,p_*]}
		\end{equation}
		where $\det$ is the determinant and $p_*$ is an arbitrary point in $H$ not collinear 
		with 
		$p_u,p_v$. Expression~\eqref{eq:grad_force} states that the difference in gradients 
		between adjacent lifted faces is equal to the force within the connecting edge.
	\end{theorem}
	\begin{proof}
		The proof is similar to Theorem~\ref{thm:plane2D} and follows from the
		exact sequence~\eqref{eq:affine_1}. Any self stress $w\in H_1 \calf$
		has a preimage $h\in H_2 \cala$ by exactness, and between
		any two preimages $h,k$ of $w$, the difference $h-k$ has the preimage of a
		global affine function in $A\R^2$. Cycles of $ H_2 \overline{A\R^2} \iso A\R^2$ are 
		assignments of the same affine function to all cells. The polyhedral lifts $h,k$ differ by 
		an affine function everywhere.
		
		The value of edge forces $w_e$ follows from the definition of the connecting 
		homomorphism $H_2 \cala \to	H_1 \calf$. Applying the cosheaf boundary map 
		$\bdd_{\overline{A\R^2} }$ to the preimage of $h$ in $C_2 \overline{A\R^2}$
		results in an element $l\in C_1 \overline{A\R^2}$ with terms $l_e =
		[e:f] (h_f-h_g)$ over each edge with incident cells $u,v\lhd e\lhd f,g$. Over such an 
		edge, the affine function $h_f - h_g$ is zero on points $p_u,p_v$ so $l$ is an element 
		of $C_1 \calz$.
		
		Noting that each stalk $\calz_e$ is one dimensional spanned by the function
		$(p_u \times p_v)^\vee$; it follows that the scalar $w_e$ satisfies $l_e = 
		w_e(p_u\times p_v)^\vee$. For an arbitrary choice of $p_*\in H$ not
		co-linear with $p_u,p_v$, note that $(p_u \times p_v)^\vee(p_*) =
		\det[p_u,p_v,p_*] \neq 0$. Evaluating for $w_e$ we find
		\begin{equation}
			w_e = \frac{l_e(p_*)}{\det[p_u,p_v,p_*]} = [e:f]\frac{h_f(p_*) - 
				h_g(p_*)}{\det[p_u,p_v,p_*]}
		\end{equation}
		is a constant scalar independent of $p_*$. Thus
		\begin{equation}
			l_e = [e:f]\frac{h_f(p_*) - h_g(p_*)}{\det[p_u,p_v,p_*]} (p_u \times
			p_v)^\vee
		\end{equation}
		over each edge. Finally, applying the cosheaf isomorphism
		$\tau^{-1}(l) = w$ results in the desired self-stress.
	\end{proof}
	
	The gradient of an affine function $h_f$ over the plane $H\subset \R^3$ consists of the 
	first two coefficients only, the last being a constant term. These two coefficients can be 
	considered as {\it coordinates} of the dual vertex $\tilde{f}$ in $H$. These coordinates 
	assemble into a {\it perpendicular} reciprocal diagram. This outlines an alternative way 
	of 
	generating dual figures; starting from the space of polyhedral lifts of $(X,p)$ there is an 
	injective map to the space of perpendicular reciprocal diagrams $(\tilde{X},q)$ over any 
	cellular oriented manifold \cite{Crapo1993}.
	
	\subsection{Boundary Conditions}
	\label{sec:boundary}
	In prior discussions, a truss is modeled by a cell complex $X$ with the specific topology 
	of 
	the sphere $S^2$. Because a sphere has empty boundary it was not 
	possible to model {\it boundary forces} that may be imparted on the truss. These 
	boundary conditions consist of external or reaction loads imparted onto the structure. 
	With the vast majority of trusses in practice are subject to external loads, an alternative 
	cellular topology must be selected.
	
	External loads acting on a truss are modeled by edges $e_*$ connected to the truss 
	only at one end $v\lhd e_*$ along the exterior of the truss. Here $v$ is the point of 
	application of the foreign force, applied along the line of action of $e_*$. These edges 
	may be external loads $e_l$ or reaction forces $e_r$; both are treated equally as 
	partially open ended edges. With this, the truss including these open ended edges $X$ 
	has the topology of an open disk $B^2\subset S^2$.
	
	Considering boundary conditions, the homology $H_1 \calf$ of the force cosheaf still 
	describes the space of canceling edge forces. These include self stresses as well as 
	equilibrium stresses induced by boundary loads. If the values of external loads $e_l$ are 
	provided beforehand, solving for internal and reaction forces is equivalent to finding a 
	cycle in $H_1 \calf$ with the specified values over edges $e_l$.
	
	Graphic statics readily generalizes to the setting of boundary conditions. The most 
	important difference is that the open disk $B^2$ has (Borel-Moore) homology $H_2 B^2 
	\iso \R$, $H_1 B^2 = 0$ and $H_0 B^2 = 0$. The key difference to spherical topology is 
	that $H_0$ of $X$ (and $H_0$ of constant cosheaves $\overline{\R^n}$) vanish.
	
	Algebraic relations in graphic statics earlier in Section~\ref{sec:2D1} were found to be
	governed by the two exact sequences~\eqref{eq:2D1} and~\eqref{eq:2D2}. When $X$ 
	takes the topology of the open disk, the former sequence remains the same while the 
	latter takes the form
	\begin{equation}
		0 \to H_1 \calg \to H_0 \calf \to 0.
	\end{equation}
	In other words, the space of degrees of freedom $H_0 \calf$ is isomorphic to dual edge 
	rotations $H^1 \tilde{\calg}$ {\it without} needing to account for the global translations. 
	External edges $e_l$ and $e_r$ may lock the truss $(X,p)$ in place, so free translations 
	may not exist at all. Apart from this consideration of global translation,
  Theorems~\ref{thm:plane2D} and~\ref{thm:second_theorem} remain valid in the
  presence of boundary conditions.
	
	In the setting of plane reciprocity the form diagram and the force diagram are no longer 
	symmetrically dual. When $X$ has the topology of an open disk, its dual $\tilde{X}$ has 
	the dissimilar topology of a {\it closed} disk. This is {\it Lefschetz duality} over subsets of 
	manifolds, here $S^2$. If $(X,p)$ is an open form diagram with boundary conditions and 
	$(\tilde{X},q)$ is a closed force diagram, this topological change eliminates the need to 
	reduce by the translation factor of $\R^2$ in points (i) and (ii) in 
	Theorem~\ref{thm:plane 
		reciprocity}.
	
	For the polyhedral lifting correspondence, Theorem~\ref{thm:maxwell_lift} is 
	unchanged over a truss with boundary conditions, as Sequence~\ref{eq:affine_1} does 
	not consider zero dimensional homology. In effect, polyhedral lifts can follow from 
	equilibrium loading conditions as well.
	
	\begin{figure}
		\centering
		\begin{subfigure}[t]{0.45\textwidth}
			\centering \includegraphics{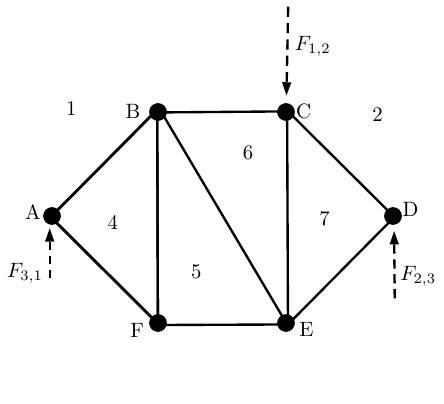}
		\end{subfigure}
		\hspace{0.01\textwidth}
		\begin{subfigure}[t]{0.45\textwidth}
			\centering \includegraphics{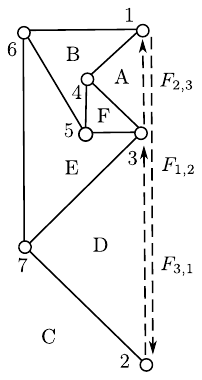}
		\end{subfigure}
		\caption{A truss with boundary conditions $X$ with one dimension of stress is 
		pictured 
			(left). There is one dimension of non-trivial realization of the dual $\tilde{X}$ 
			(right).}
		\label{fig:planar_relative}
	\end{figure}
	
	\begin{figure}[ht]\centering
		\includegraphics{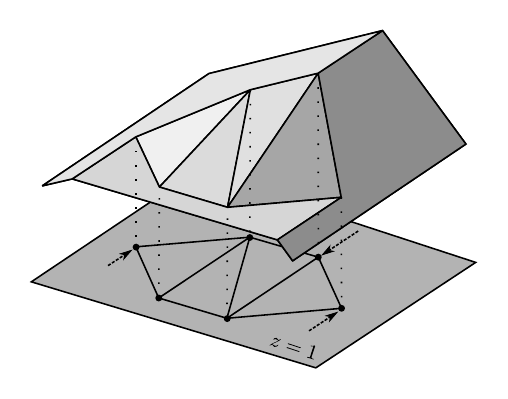}
		\caption{A lifting of the truss with boundary conditions in
			Figure~\ref{fig:planar_relative}.}
		\label{fig:affine_airy}
	\end{figure}
	
	%
	\section{Beyond Planar Graphic Statics} 
	\label{sec:nonsphere}
	%
	There is the natural question of how to extend graphic statics to non-spherical 
	topology, 
	allowing for dual diagrams from non-planar trusses. Previous approaches have focused 
	on dualizing the underlying graph, not the full oriented manifold $M_g$ of genus $g$ as 
	we will here \cite{GeneralizedMicheletti2008, AlgebraicVanMele2014}.
	
	\begin{theorem}[Higher Genus Plane Graphic Statics]\label{thm:genus}
		Let $(X, p)$ be a diagram with the topology of an oriented surface $M_g$ realized in 
		$\R^2$. If the dimension of axial self stress $\dim H_1 \calf$ is 
		greater than $4g$, then there exists a non-trivial parallel dual realization $q$ of
		$\tilde{X}$ in $\R^n$
	\end{theorem}
	\begin{proof}
		For $\calf$ in ambient dimension $n$, the long exact sequence for the
		filtration $\calf \to \overline{\R^2}$ is
		\begin{equation}\label{eq:genus 1}
			0\to \R^2 \to H_2 \overline{\R^2}/\calf \xrightarrow{\vartheta} H_1 \calf 
			\xrightarrow{\phi} H_1 \overline{\R^2} \to \dots
		\end{equation}
		Since $\dim H_1 M_g = 2g$, it follows that $H_1 \overline{\R^2}\iso \R^{4g}$.
		Thus under the assumptions the map $\phi$ above cannot be injective. By exactness 
		of sequence~\eqref{eq:genus 1}, the connecting homomorphism $\vartheta$ is not 
		the zero map. Any non-trivial preimage by $\vartheta$ correlates to a non-trivial 
		parallel realization of $\tilde{X}$.
	\end{proof}
	
	We take care to note that the resulting realizations of $\tilde{X}$ are not
	the algebraic duals in graph theory, where simple cycles are dual to minimal
	cuts. Hence, the above theorem does not contradict Whitney's duality
	\cite{Whitney1931}, which states that a graph has a dual if and only if it
	is planar. Instead, Theorem~\ref{thm:genus} is a statement strictly about geometric 
	duals, linking vertices to dual faces, edges to dual edges, and faces to dual vertices.
	
	Another point is that a realization of the torus $M_1$ is a mapping to $\R^2$ and is not 
	a periodic tiling of the plane, as is the case in \cite{borcea2015liftings,
		erickson2020toroidal}.
	
	\begin{example}[Dependence on Graph Embedding]\label{ex:torus_dual}
		Finding dual realizations from a non-planar graph depends on its embedding in an 
		oriented surface $M_g$. The particular introduction of face cells determines what 
		dual
		realizations result from a self-stress. Even for planar graphs embedded in the sphere 
		$S^2$, the choice of realization does not necessarily determine the dual graph (unless 
		the	graph is 3-connected \cite{Whitney1933}).
		
		To demonstrate, let $(X,p)$ be the cellular decomposition of the torus $M_1$ 
		pictured
		in Figure~\ref{fig:nonplanar}(a) where it contains the complete bipartite
		graph $K_{3,3}$ as a subgraph. The realization of $X$ in $\R^2$ pictured in
		(b) has a non-planar 1-skeleton and four dimensions of self stress in $H_1 \calf_X$. 
		There 
		is one nontrivial dimension of dual realizations of $\tilde{X}$, pictured in (c). The 
		1-skeleton of $\tilde{X}$ is planar, but there is no way to embed faces of $\tilde{X}$ 
		in $\R^2$ without overlapping.
		
		This process outlines a canonical way to dualize a cell complex realized in
		$\R^2$, but not a canonical way to dualize the underlying graphs. The cell complex 
		$X'$ 
		in (d) has a homeomorphic 1-skeleton to (a), but with $(X',p)$ utilizing the same 
		realization (b), the space $H^0 \tilde{\calg}_{\tilde{X}'}$ and all dual realizations of 
		$\tilde{X}'$ are trivial.
	\end{example}
	
	\begin{figure}[!ht]
		\begin{subfigure}[t]{0.50\textwidth}\centering
			\includegraphics{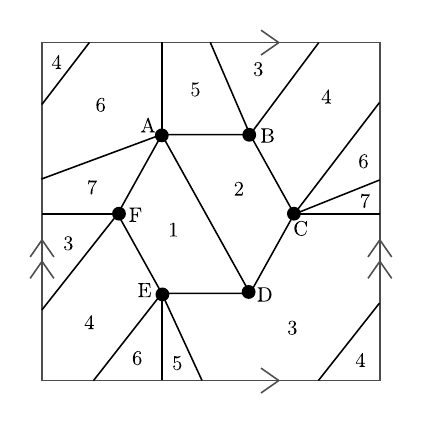}
			\caption{Topology of a cellular decomposition $X$ of $M_1$ containing
				$K_{3,3}$, a non-planar graph.}
		\end{subfigure}
		\begin{subfigure}[t]{0.40\textwidth}\centering
			\includegraphics{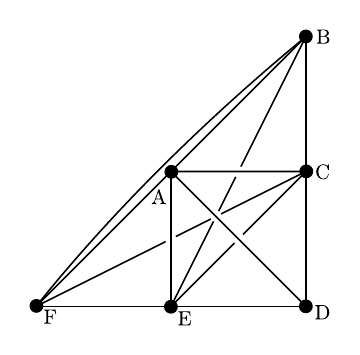}
			\caption{A realization of $X$ in the plane.}
		\end{subfigure}
		
		\begin{subfigure}[t]{0.40\textwidth}\centering
			\includegraphics{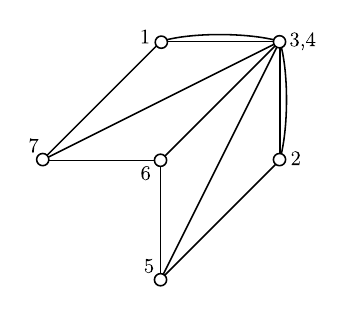}
			\caption{A realization of a dual diagram $\tilde{X}$ to that in (b).}
		\end{subfigure}
		\begin{subfigure}[t]{0.50\textwidth}\centering
			\includegraphics{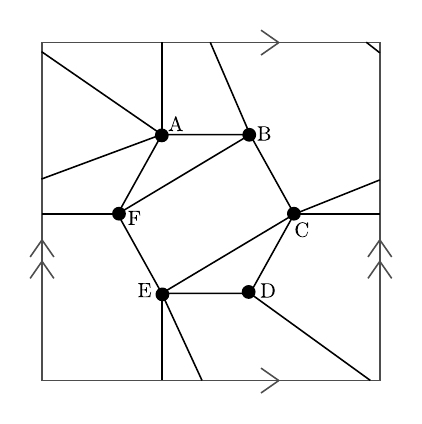}
			\caption{An alternate cellular decomposition $X'$ of $M_1$ with the same
				1-skeleton as (a).}
		\end{subfigure}
		\caption{Graphic Statics duality with non-planar graphs}
		\label{fig:nonplanar}
	\end{figure}
	
	There is an analogue to Theorem~\ref{thm:genus} for polyhedral lifts over higher genus 
	manifolds.
	
	\begin{theorem}[Higher Genus Polyhedral Lifting]\label{thm:genuslift}
		Let $(X,p)$ be a cellular decomposition of an oriented surface $M_g$ realized in 
		$\R^2$. If the dimension of the space of self stresses $H_1 \calf$ is greater than $6g$, then 
		there 
		exists a self stress of $X$ and a non-trivial polyhedral lifting to $\R^3$, where the 
		difference in gradient between adjacent lifted faces is equal to the force within the 
		connecting edge.
	\end{theorem}
	\begin{proof}
		Recall that the space of affine functions $A\R^2$ is three dimensional. Since $\dim 
		H_1 
		M_g = 2g$, it follows that $H_1 \overline{A\R^2}\iso \R^{6g}$. If $H_1 \calf$ has 
		dimension greater than $6g$, the map $\phi: H_1 \calf \to H_1 \overline{A\R^2}$ 
		cannot be injective. By exactness of sequence~\eqref{eq:genus 1}, the connecting 
		homomorphism $\vartheta$ cannot be the zero map. Any non-trivial preimage by 
		$\vartheta$ 
		correlates to a non-trivial polyhedral lift.
	\end{proof}
	
	A triangulated surface in general position has non-trivial 
	polyhedral lifts. In a triangular mesh every node is incident to only triangular faces, and 
	a lift of the node lifts these triangles into a cone. Consequently, a regular cellular 
	decomposition of a manifold into triangles results in a $|V|$  dimensional space of 
	polyhedral lifts, where $|V|$ is the number of vertices.
	
	\begin{figure}[!ht]
		\begin{subfigure}[t]{0.45\textwidth}\centering
			\includegraphics{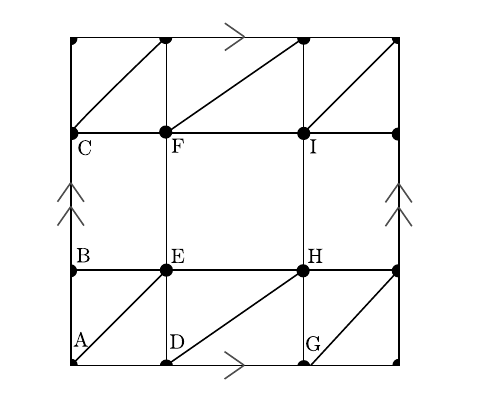}
			\caption{A cell complex $X$ with topology $M_1$.}
		\end{subfigure}
		\begin{subfigure}[t]{0.45\textwidth}\centering
			\includegraphics{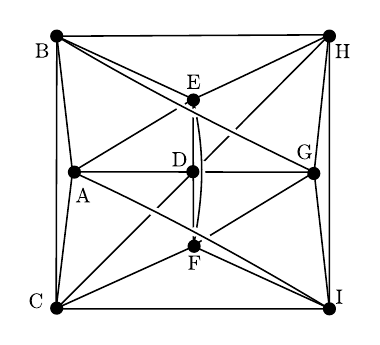}
			\caption{A form diagram $(X,p)$ in the plane with 9 degrees of self stress.}
		\end{subfigure}
		\centering
		\begin{subfigure}[t]{0.65\textwidth}\centering
			\includegraphics{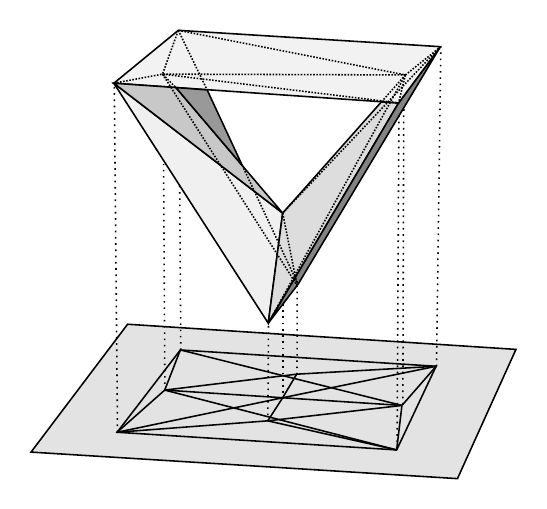}
			\caption{A polyhedral lift of $X$ to $\R^3$.}
		\end{subfigure}
		\caption{A cell complex $X$ with topology (a) is realized in (b). The truss is rigid and 
		by 
			the Euler characteristic~\eqref{eq:euler_force}, we observe that there are $H_1 
			\calf = 
			24 - 18 + 3 = 9$ degrees of self stress. With $9>6$, Theorem~\ref{thm:genuslift} 
			applies and there is at least one non-trivial dimension of torus polyhedral lifting in 
			$H_2 \cala / \overline{A\R^2}$, with a representative pictured in (c).	Numerical 
			computations indicates that there are in fact four dimensions of lifts up to shifts by 
			global affine functions.}
		\label{fig:torus_lift}
	\end{figure}
	
	In planar graphic statics, Theorems~\ref{thm:plane2D} and \ref{thm:maxwell_lift} 
	prove that a structure has an non-trivial reciprocal diagram if and only if the structure 
	has a non-trivial polyhedral lifting. Over a torus and higher genus surfaces, this 
	relationship no longer holds true \cite{Crapo1993}. Polyhedral lifts guarantee reciprocal 
	figures but the converse does not hold; Figure~\ref{fig:torus_pentagon} portrays a
	counterexample.
	
	\begin{figure}[!ht]\centering
		\begin{subfigure}[t]{0.45\textwidth}\centering
			\includegraphics{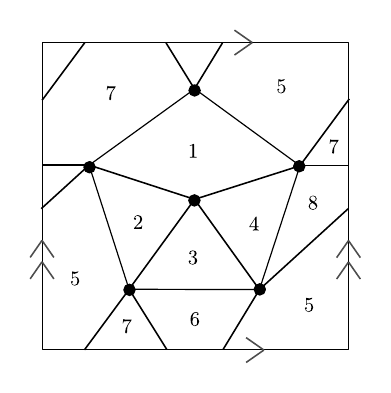}
			\caption{Toroidal topology of $X$.}
		\end{subfigure}
		\begin{subfigure}[t]{0.45\textwidth}\centering
			\includegraphics{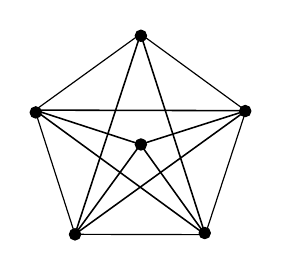}
			\caption{A form diagram $(X,p)$ with 5 degrees of self stress.}
		\end{subfigure}
		
		\begin{subfigure}[t]{0.65\textwidth}\centering
			\includegraphics{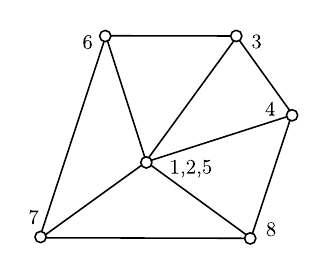}
			\caption{A dual force diagram is generated.}
		\end{subfigure}
		\caption{An example of a toroidal figure with reciprocal diagram but no non-trivial 
		polyhedral lift. In the form diagram (b), faces 5 and 1 lock all vertices in a single plane 
		preventing lifts. There are 5 degrees of self stress, meaning a non-trivial force 
		diagram exists by Theorem~\ref{thm:genus} in (c).}
		\label{fig:torus_pentagon}
	\end{figure}
	
	\begin{figure}[!ht]\centering
		\includegraphics[scale=0.8]{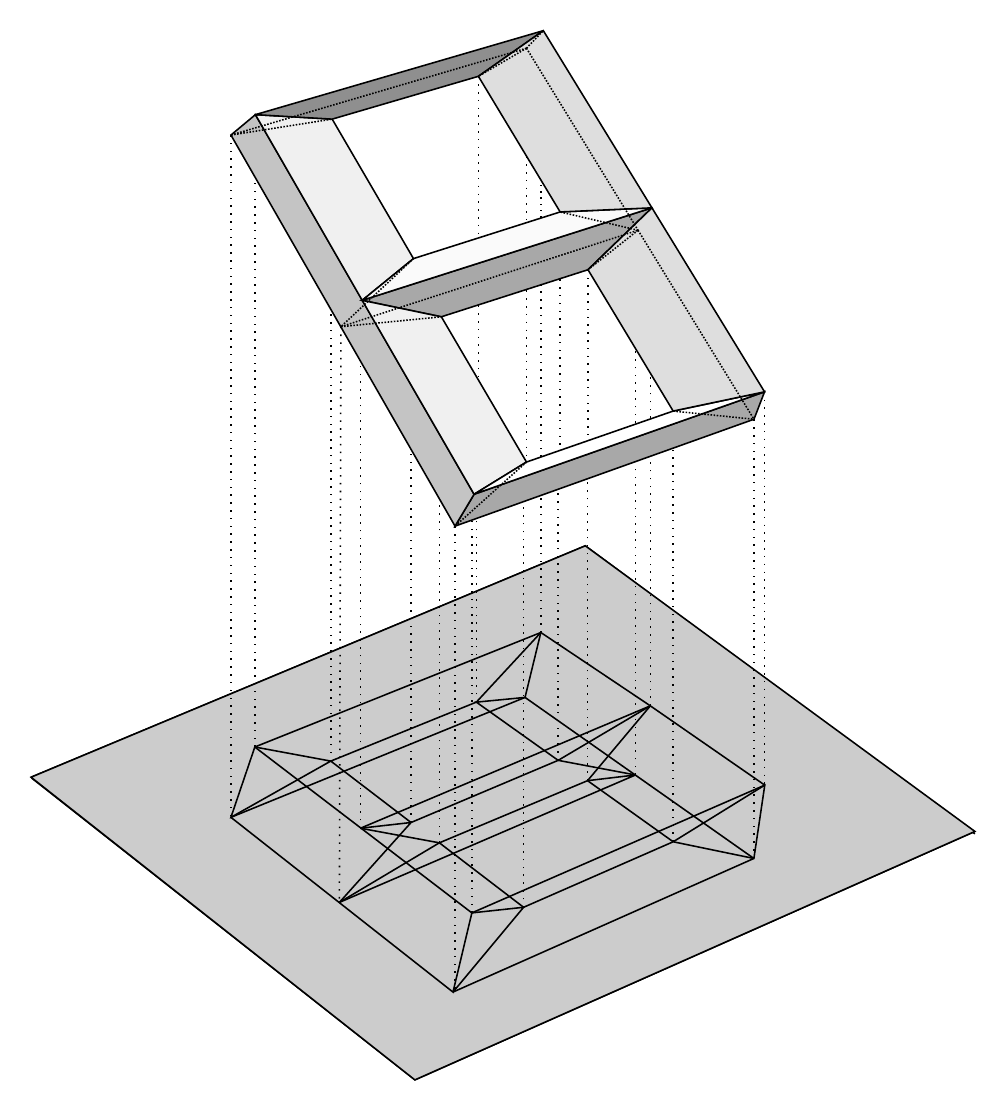}
		\caption{A polyhedral lift of a cell complex with the topology of a two holed torus 
			$M_2$ is pictured. Numerical computation shows this particular realization has 8 
			degrees of self stress, which is less than the 13 degree requirement to guarantee 
			the 
			existence of a lift. Due to the specific geometry of the realization there are 3 
			dimensions of polyhedral lifts up to shift 
			by affine function. One dimension determines width of the bounding box, the 
			other two determine the vertical position of the rectangular voids.}
	\end{figure}
	
	%
	\section{Conclusion}
	\label{sec:conc}
	%
	We proposed a development of algebraic graphic statics following sheaf and cosheaf 
	constructions. These formal methods from algebraic topology allow for extensions to 
	high 
	genus surfaces. We precisely defined reciprocal diagrams in the plane and found $H_1$ 
	of 
	the base surface to obstruct their formation. These homological constraints impede the 
	formation of lifts or dual diagrams in the plane. There are numerous future avenues of 
	research.
	
	\subsection{Sheaf Finite Elements}
	The {\it finite element method} is central to modern structural analysis, fluid flow, and 
	heat transfer mathematical modeling. In summary, a complex domain $M$ is 
	subdivided and discretized into a finite number of geometrically simple elements 
	\cite{reddy2019introduction}. A solution to a Poisson problem with boundary 
	conditions 
	over $M$ is approximated by algebraic approximation functions between the finite 
	elements, these derived from weighted integrals over local 
	neighborhoods. These approximation functions are typically polynomials and the entire 
	discrete system is otherwise known as a spline \cite{billera1988homology}. Introduced 
	approximation errors introduced can be bounded and shown (in a well formulated 
	model) 
	to converge to 0 as the mesh is refined.
	
	It has been shown that polynomial splines are homological and can be 
	described by cosheaves \cite{GhristEAT}. In this setting the force cosheaf $\calf$ is a 
	spline of constant functions and the affine cosheaf $\cala$ is a spline of linear-affine 
	functions. The stiffness matrix problem $a = \Delta_\calf b$ and its pseudo-inverse are 
	quintessential in linear-elastic mechanics. There is potential to extend these 
	homological 
	methods to other finite element processes.
	
	\subsection{Trusses on Manifolds}
	Definition~\ref{def:forcecosheaf} of the force cosheaf suggests an extension of truss 
	mechanics to general manifolds. Instead of a cell complex 
	$X$ being realized in Euclidean space, we realize vertices $v$ with positions $p_v \in 
	M$ and where edges are geodesics $\gamma_e: \R^1 \to M$ (one must take care that 
	these geodesics exist e.g. by the Hopf--Rinow theorem). To an edge $u,v\lhd e$, the 
	force cosheaf $\calf$ would have extension maps $\calf_{e \rhd u}$ and 
	$\calf_{e \rhd v}$ the push-forward $(\gamma_e)_*$ mapping to $T_{p_u} M$ and 
	$T_{p_v} M$ along the curve. A more complete analysis of manifold trusses 
	warrants further attention.
	
	\subsection{3D Graphic Statics}
	The homological constructions here are sufficient to generalize to higher dimensions. In 
	{\it 3D vector graphic statics} the form diagram is embedded 
	in $\R^3$. In Section~\ref{sec:2D1} we can replace all occurrences of $\R^2$ with 
	$\R^3$, 
	leading to 3D vector graphic statics duality. Both Theorems~\ref{thm:plane2D} and 
	Theorems~\ref{thm:second_theorem} can be translated. As the technique of rotating 
	stalks by one quarter turn does not transfer there is no notion of reciprocity.
	
	Another avenue is in {\it 3D polyhedral graphic statics} which has shown to 
	have geometric relations far more complex than the 2D setting 
	\cite{ReciprocalAkbarzadeh2016}. Investigations show that the linear algebraic relations 
	when including 3-cells are described by a {\it spectral sequence} of pertinent cosheaves.
	
	\newpage
	\bibliographystyle{IEEEtran} {\small \bibliography{Graphic_Statics}}
	
\end{document}